\documentclass[preprint,review,12pt]{amsart}

\usepackage{amsfonts,amssymb,amsmath,amsthm}
\usepackage[cp1251]{inputenc}
\usepackage[T2B]{fontenc}
\usepackage[english]{babel}
\usepackage[mathscr]{eucal}
\usepackage{calrsfs}

\newtheorem{theorem}{Theorem}[section]

\newtheorem{proposition}[theorem]{Proposition}
\newtheorem{corollary}[theorem]{Corollary}

\theoremstyle{definition}
\newtheorem{definition}[theorem]{Definition}
\newtheorem{example}[theorem]{Example}

\theoremstyle{remark}
\newtheorem{remark}[theorem]{Remark}

\numberwithin{equation}{section}
\usepackage[left=2.0cm,right=2.2cm,
    top=2.5cm,bottom=2.5cm]{geometry}

\DeclareMathOperator*{\esup}{ess\,sup}

\begin{document}

\title[Hardy--Steklov operators and embedding inequalities]{Hardy--Steklov operators and embedding inequalities of Sobolev type
}

\author[M.G. Nasyrova and E.P. Ushakova]{Maria G. Nasyrova         and
        Elena P. Ushakova }

\address{Computing Center of the Far Eastern Branch of the Russian Academy of Sciences, Kim Yu Chena 65, 680000 Khabarovsk, Russia            
              }
     
\email{elenau@inbox.ru, nassm@mail.ru}

\maketitle

\begin{abstract}
We characterize a weighted norm inequality which corresponds to the embedding of a class of absolutely continuous functions into the fractional
order Sobolev space. The auxiliary result of the paper is of independent interest. It comprises of several types of necessary and sufficient conditions
for the boundedness of the Hardy--Steklov operator (an integral operator with two variable boundaries of integration) in weighted Lebesgue spaces.


\end{abstract}

\section{Introduction}
\label{intro}
Given $s>0$ and a weight function $v\ge 0$ on $(0,\infty)$ let $L^s_v(0,\infty)=:L_v^s$ denote the weighted Lebesgue space of all measurable
functions $g$ on $(0,\infty)$ satisfying $$\|g\|_{L_v^s}:=\|g\|_{s,v}:=\biggl(\int_0^\infty|g(t)|^sv(t)\,\mathrm{d}t\biggr)^{\frac{1}{s}}<\infty.$$
Suppose that $p\ge 1$, $q>0$, $\lambda\in(0,1)$ and a weight function $u$ is non--negative on $(0,\infty)\times(0,\infty)$. Consider the fractional inequality
\begin{equation}\label{00}
\biggl(\int_0^\infty\!\!\!\!\int_0^\infty\frac{|f(x)-f(y)|^q}{|x-y|^{1+\lambda q}}\,u(x,y)\,\mathrm{d}x\,\mathrm{d}y\biggr)^{\frac{1}{q}}\le C\biggl(
\int_0^\infty |f'(z)|^pv(z)\,\mathrm{d}z\biggr)^{\frac{1}{p}}
\end{equation} 
reflecting the embedding of a subclass $$\mathcal{W}_{p,v}^1:=\{f\in AC(0,\infty): \|f'\|_{p,v}<\infty\}$$ of absolutely
continuous functions (AC--functions) on $(0,\infty)$ into the weighted fractional order Sobolev
space $$W_{q,u}^\lambda:=W_{q,u}^\lambda(0,\infty) := \Bigl\{f :
[f]_{W_{q,u}^\lambda}<\infty\Bigr\}.$$ Given $f\in W_{q,u}^\lambda$, the functional
$$[f]_{W_{q,u}^\lambda}:=\biggl(\int_0^\infty \!\!\!\!\int_0^\infty\frac{|f(x)-f(y)|^q}{|x-y|^{1+\lambda
q}}\,u(x,y)\,\mathrm{d}x\,\mathrm{d}y\biggr)^{\frac{1}{q}}$$ is the weighted Slobodeckij type
seminorm. For $q\ge 1$ the space $W_{q}^\lambda:=W_{q,1}^\lambda$, equipped with the norm
$\|f\|_{W_{q,1}^\lambda}=\|f\|_{q,1}+[f]_{W_{q,1}^\lambda}<\infty$, can also be called Aronszajn
space \cite{Ar}, Gagliardo space \cite{Gar} or Slobodeckij space \cite{Sl}. Being a special case of
Besov spaces [\ref{BIN}, \ref{Tri}] $W_q^\lambda$ plays an important role in the study of traces of Sobolev
functions and are applicable  to partial differential equations (see \cite{LM} for details).

The inequality (\ref{00}) was studied by H. P. Heinig and G. Sinnamon in \cite{HS}, where some conditions for the validity of (\ref{00}) were found. The method, employed in \cite{HS} for the investigation of (\ref{00}), was based on characterization of the Hardy--Steklov
operator
\begin{equation}\label{1}\mathcal{H}g(x)=\int_{a(x)}^{b(x)}g(y)\,\mathrm{d}y\end{equation}
with the boundaries $a(x)$ and $b(x)$ satisfying the conditions:
\begin {equation}
\label {3}~\hspace{-.5cm}
\begin {tabular}{ll} (i) &
$a(x),$ $b(x)$ are differentiable and strictly increasing on
$(0,\infty);$\\ (ii) & $a(0)=b(0)=0,$ $a(x)< b(x)$ for
$0<x<\infty,$ $a(\infty)=b(\infty)=\infty.$
\end {tabular}
\end {equation} The Hardy--Steklov operator (\ref{1}) was studied in [\ref{BS}, \ref{CS}, \ref{HS}, \ref{GL}, \ref{SU}, \ref{SU2}]. It has connections with
other integral transformations [\ref{Harper}, \ref{HS}, \ref{KP}, \ref{SUmia}, \ref{U2}], some embedding theorems
\cite{O2011} and is applicable to other neighbour areas \cite[Ch. 3]{KP}.

The purpose of this paper is twofold. Firstly, we obtain new necessary and sufficient conditions
for the boundedness of $\mathcal{H}$ from $L_v^p$ to $L_w^q$, when $p>1$ and $q>0$  (see Theorem 2.2 and Corollary 2.3). These
conditions complement the already existed ones (see e.g. \cite[Th. 4.1]{SUmia}) with new facts and give
alternative characteristics for the Hardy--Steklov operator $\mathcal{H}$ in weighted Lebesgue spaces. Secondly, using
the idea by H. P. Heinig and G. Sinnamon from \cite[\S~3]{HS} and our new characterizations of 
$\mathcal{H}$, in particular Corollary 2.3, we find conditions for the inequality
(\ref{00}) to hold for all $p\ge 1$ and $q>0$ (see Theorem 3.1). The conditions found in Theorem 3.1 complement the earlier results by H. P. Heinig and G. Sinnamon \cite[Th. 3.3 and Cor. 3.5]{HS} with the case $0<q<p<\infty$, which was not considered in \cite{HS}. For $1\le p\le q<\infty$ we derive characteristics alternative to those given in \cite{HS}.

Throughout the paper products of the form $0\cdot\infty$ are
taken to be equal to 0. Relations of the type $A\ll B$ mean that $A\le cB$ with
some constant $c$ depending, possibly, on parameters $p$ and $q$ only. We write
$A\approx B$ instead of $A\ll B \ll A$ or $A=cB$. We use $\mathbb Z$ and $\mathbb N$ for integers and natural numbers, respectively.
$\chi_E$ stands for the
characteristic function (indicator) of a subset $E\subset (0,\infty)$. We make use of marks $:=$ and
$=:$ for introducing new quantities. 
We assume weight functions to be non--negative and locally integrable to  appropriate powers. In order to shorten big formulae, we use $\int h$ to cut $\int h(t)\,\mathrm{d}t$ for a one--variable integrand $h(t)$, where it makes sence.

\section{Hardy--Steklov operator}
\label{sec:1}
Let $p>1$, $q>0$ and $v,w$ be weight functions on $(0,\infty)$. Suppose that the boundaries $a(x)$ and $b(x)$ of the operator \eqref{1} satisfy
the conditions \eqref{3} and $a^{-1}(y)$, $b^{-1}(y)$ are the inverse functions to $y=a(x)$ and $y=b(x)$.

The most recent development, related to the characterization of the operators \eqref{1} with boundaries $a(x)$ and $b(x)$ satisfying \eqref{3},
is based on the conception of fairway. The fairway $\sigma(x)$ is a function on $(0,\infty)$ with some special properties. It was introduced in
\cite{SU} in order to derive new forms of boundedness criteria for $\mathcal{H}$ acting in weighted Lebesgue spaces, say, from $L_v^p$ to
$L_w^q$. {\it The fairway--function} $\sigma(x)$ in \cite{SU} (see also \cite[Def. 2.4]{SUmia}) is built on given boundaries $a(x)$ and
$b(x)$, weight function $v\ge 0$ and parameter $p>1$ so that $a(x)<\sigma(x)<b(x)$ and
\begin{equation}\label{10v}\int_{a(x)}^{\sigma(x)}v^{1-p'}(y)\,\mathrm{d}y=\int_{\sigma
(x)}^{b(x)} v^{1-p'}(y)\,\mathrm{d}y\qquad{\rm for}\;\;{\rm
all}\;\;x>0.\end{equation}

The conception of fairway $\sigma$ brought a collection of results for $\mathcal{H}:L_v^p\to L_w^q$, which appeared to be convenient to
further development and applications (see e.g. [\ref{L1}, \ref{L2}, \ref{SU4}, \ref{SU2}, \ref{SUmia}, \ref{U2}]. In particular, two different types of the boundedness criteria
for \eqref{1} have been obtained in terms of the fairway $\sigma$ (see e.g. \cite[Ths. 4.1, 4.2]{SUmia} or \cite[Th. 2.1]{SU3}) and
applied to neighbour problems: e.g. characterization of the geometric mean operator with variable boundaries in Lebesgue spaces, weighted embedding inequalities of
Sobolev type (see Th. 5.1 and \S~6 in \cite{SUmia} for details).

The dual form of the boundedness criteria for $\mathcal{H}:L_v^p\to L_w^q$ was stated in \cite[Th. 2.2]{SU3} basing on the definition of
the dual fairway--function.
\begin{definition}\label{def1} Given boundary functions $a(x)$ and
$b(x),$ satisfying the conditions \eqref{3}, and a weight function $w(x)$ such that
$0<w(x)<\infty$ a.e. $x\in(0,\infty)$ and $w(x)$ is locally integrable on
$(0,\infty),$ we define the {\it dual fairway--function} $\rho
(y)$ such that $b^{-1}(y)< \rho(y)< a^{-1}(y)$ on $(0,\infty)$ and \begin{equation}
\label{10} \int_{b^{-1}(y)}^{\rho(y)}w(x)\,\mathrm{d}x=\int_{\rho
(y)}^{a^{-1}(y)} w(x)\,\mathrm{d}x\qquad{\rm for}\;\;{\rm
all}\;\;y>0.
\end{equation}\end{definition}

Similarly to $\sigma(x)$, the dual fairway $\rho(y)$ is differentiable and strictly increasing
function on $(0,\infty)$. The result of Theorem 2.2 from \cite{SU3} follows, by duality, from
\cite[Th. 2.1]{SU3} but is valid for $p>1$ and $q>1$ only.

In this section we prove that the dual form of the boundedness criteria for \eqref{1} from $L_v^p$ to $L_w^q$ 
is also true when $p>1$, $0<q<1$ (see Theorem 2.2 (c), (e)) and can be obtained in terms of the original fairway--function $\sigma$ for all
$p>1$, $q>0$ as well. As a counterpart of this statement, we show that the boundedness criteria in their original (non--dual) form can be
obtained with help of the dual fairway--function $\rho$ instead of $\sigma$. In comparison with all the earlier results for Hardy--Steklov
operator \eqref{1} in weighted Lebesgue spaces (see e.g. \cite{SUmia}) the sufficient parts of 
our new boundedness criteria for $\mathcal{H}:L_v^p\to L_w^q$ do not require the condition \eqref{10v} on $\sigma$ or the condition
\eqref{10}  on $\rho(y)$ (see Theorem 2.2 (a), (b)) if $p,q>1$.

Having several forms of characteristics for an operator in function spaces may promote solutions to some neighbour problems.
This argument concerns directly the embedding inequality \eqref{00}, characterization of which is the main problem of our paper.
The inequality \eqref{00} is connected to the Hardy--Steklov operator (see Section 2 for details) and, depending on the weights $v$ and $w$,
extracting $\rho$ from the equation \eqref{10} may be more practical 
than finding $\sigma$ from \eqref{10v}, and vice versa.

 \smallskip

Given boundaries $a(x)$ and $b(x)$ satisfying the conditions \eqref{3} let $y=\varsigma(x)$ be a strictly increasing continuous function on
$(0,\infty)$ such that $a(x)< \varsigma(x)< b(x)$ for $x>0$. Then the inverse $x=\varsigma^{-1}(y)=:\varsigma^\ast(y)$ is also a strictly increasing
continuous function on $(0,\infty)$ such that $b^{-1}(y)< \varsigma^{\ast}(y)< a^{-1}(y)$ if $y>0$. For such a function $\varsigma$ we
put 
\begin{align*} \Delta(t):=&\Delta^-(t)\cup\Delta^+(t), & \Delta^-(t):=&[a(t),\varsigma(t)), &\Delta^+(t):=&[\varsigma(t),b(t)),\\
\delta(t):=&\delta^-(t)\cup\delta^+(t),
&\delta^-(t):=&[b^{-1}(\varsigma(t)),t),
&\delta^+(t):=&[t,a^{-1}(\varsigma(t))),\\
\Theta(t):=&\Theta^-(t)\cup\Theta^+(t), 
&\Theta^-(t):=&[b^{-1}(t),\varsigma^{\ast}(t)), 
&\Theta^+(t):=&[\varsigma^{\ast}(t),a^{-1}(t)),\\
\vartheta(t):=&\vartheta^-(t)\cup\vartheta^+(t), 
&\vartheta^-(t):=&[a(\varsigma^{\ast}(t)),t), 
&\vartheta^+(t):=&[t,b(\varsigma^{\ast}(t))).\end{align*} Let $W_I(t):=\int_{I(t)}
w 
$ and $V_I(t):=\int_{I(t)}v^{1-p'}$ for some $I(t)\subset(0,\infty)$.
For $p>1$ put $ p':=p/(p-1)$, $q':=q/(q-1)$, $r:=pq/(p-q)$ and denote $\varrho(t):=v^{1-p'}(t),$\\
\begin{align*}\mathcal{A}_\varsigma:=& \sup_{t>0}\bigl[W_\delta(t)\bigr]^{\frac{1}{q}}
\bigl[V_\Delta(t)\bigr]^{\frac{1}{p'}}, &
\mathcal{B}_\varsigma:=&\biggl(\int_0^\infty
\bigl[W_{\delta}\bigr]^{\frac{r}{p}}
\bigl[V_\Delta\bigr]^{\frac{r}{p'}}w\biggr)^{\frac{1}{r}},\\
\mathcal{A}_\varsigma^\pm:=& \sup_{t>0}\bigl[W_{\delta^\pm}(t)\bigr]^{\frac{1}{q}}
\bigl[V_\Delta(t)\bigr]^{\frac{1}{p'}}, &
\mathcal{B}_\varsigma^\pm:=&\biggl(\int_0^\infty
\bigl[W_{\delta^\pm}\bigr]^{\frac{r}{p}}
\bigl[V_\Delta\bigr]^{\frac{r}{p'}}w\biggr)^{\frac{1}{r}},\\
(\mathcal{A}_{\varsigma^\ast})^\ast:=& \sup_{t>0}\bigl[W_\Theta(t)\bigr]^{\frac{1}{q}}
\bigl[V_\vartheta(t)\bigr]^{\frac{1}{p'}}, &
(\mathcal{B}_{\varsigma^{\ast}})^\ast:=&
\biggl(\int_0^\infty\bigl[W_{\Theta}\bigr]^{\frac{r}{q}}
\bigl[V_\vartheta\bigr]^{\frac{r}{q'}}
\varrho\biggr)^{\frac{1}{r}},\\
(\mathcal{A}_{\varsigma^\ast}^\pm)^\ast:=& \sup_{t>0}\bigl[W_\Theta(t)\bigr]^{\frac{1}{q}}
\bigl[V_{\vartheta^\pm}(t)\bigr]^{\frac{1}{p'}}, &
(\mathcal{B}_{\varsigma^{\ast}}^\pm)^\ast:=&\biggl(\int_0^\infty\bigl[W_{\Theta}\bigr]^{\frac{r}{q}}
\bigl[V_{\vartheta^\pm}\bigr]^{\frac{r}{q'}}
\varrho\biggr)^{\frac{1}{r}}.
\end{align*} Here $(\cdot)^\ast$ indicates the dual character of the functionals $(\mathcal{A}_{\varsigma^\ast})^\ast$, $(\mathcal{B}_{\varsigma^{\ast}})^\ast$, $(\mathcal{A}_{\varsigma^\ast}^\pm)^\ast$ and $(\mathcal{B}_{\varsigma^{\ast}}^\pm)^\ast$. Note that if $\varsigma=\sigma$ then
\begin{align*}\mathcal{A}_\sigma=& \, \sup_{t>0}\biggl(\int_{b^{-1}(\sigma(t))}^{a^{-1}(\sigma(t))}w(x)\,\mathrm{d}x\biggr)^{\frac{1}{q}}
\biggl(\int_{a(t)}^{b(t)}\varrho(y)\,\mathrm{d}y\biggr)^{\frac{1}{p'}},\\
\mathcal{B}_\sigma=&\, \biggl(\int_0^\infty
\biggl[\int_{b^{-1}(\sigma(t))}^{a^{-1}(\sigma(t))}w(x)\,\mathrm{d}x\biggr]^{\frac{r}{p}}
\biggl[\int_{a(t)}^{b(t)}\varrho(y)\,\mathrm{d}y\biggr]^{\frac{r}{p'}}w(t)\,\mathrm{d}t\biggr)^{\frac{1}{r}}, 
\\
(\mathcal{A}_{\sigma^{-1}})^\ast=&\, \sup_{t>0}\biggl(\int_{b^{-1}(t)}^{a^{-1}(t)}w(x)\,\mathrm{d}x\biggr)^{\frac{1}{q}}
\biggl(\int_{a(\sigma^{-1}(t))}^{b(\sigma^{-1}(t))}\varrho(y)\,\mathrm{d}y\biggr)^{\frac{1}{p'}},\\
(\mathcal{B}_{\sigma^{-1}})^\ast=&\,
\biggl(\int_0^\infty\biggl[\int_{b^{-1}(t)}^{a^{-1}(t)}w(x)\,\mathrm{d}x\biggr]^{\frac{r}{q}}
\biggl[\int_{a(\sigma^{-1}(t))}^{b(\sigma^{-1}(t))}\varrho(y)\,\mathrm{d}y\biggr]^{\frac{r}{q'}}
\varrho(t)\,\mathrm{d}t\biggr)^{\frac{1}{r}}.
\end{align*} If $\varsigma^\ast=\rho$ then we have
\begin{align}\label{A}\mathcal{A}_{\rho^{-1}}=&\, \sup_{t>0}\biggl(\int_{b^{-1}(\rho^{-1}(t))}^{a^{-1}(\rho^{-1}(t))}w(x)\,
\mathrm{d}x\biggr)^{\frac{1}{q}}
\biggl(\int_{a(t)}^{b(t)}\varrho(y)\,\mathrm{d}y\biggr)^{\frac{1}{p'}},\\
\mathcal{B}_{\rho^{-1}}=&\,\biggl(\int_0^\infty
\biggl[\int_{b^{-1}(\rho^{-1}(t))}^{a^{-1}(\rho^{-1}(t))}w(x)\,\mathrm{d}x\biggr]^{\frac{r}{p}}
\biggl[\int_{a(t)}^{b(t)}\varrho(y)\,\mathrm{d}y\biggr]^{\frac{r}{p'}}w(t)\,\mathrm{d}t\biggr)^{\frac{1}{r}},\\
(\mathcal{A}_{\rho})^\ast=&\, \sup_{t>0}\biggl(\int_{b^{-1}(t)}^{a^{-1}(t)}w(x)\,\mathrm{d}x\biggr)^{\frac{1}{q}}
\biggl(\int_{a(\rho(t))}^{b(\rho(t))}\varrho(y)\,\mathrm{d}y\biggr)^{\frac{1}{p'}},\\
(\mathcal{B}_{\rho})^\ast=&\,
\biggl(\int_0^\infty\biggl[\int_{b^{-1}(t)}^{a^{-1}(t)}w(x)\,\mathrm{d}x\biggr]^{\frac{r}{q}}
\biggl[\int_{a(\rho(t))}^{b(\rho(t))}\varrho(y)\,\mathrm{d}y\biggr]^{\frac{r}{q'}}
\varrho(t)\,\mathrm{d}t\biggr)^{\frac{1}{r}}.\label{B}
\end{align} Similarly, the functionals $\mathcal{A}_\sigma^\pm$, $\mathcal{B}_\sigma^\pm$, $(\mathcal{A}_{\sigma^{-1}}^\pm)^\ast$,
$(\mathcal{B}_{\sigma^{-1}}^\pm)^\ast$, $\mathcal{A}_{\rho^{-1}}^\pm$, $\mathcal{B}_{\rho^{-1}}^\pm$, $(\mathcal{A}_{\rho}^\pm)^\ast$
and $(\mathcal{B}_{\rho}^\pm)^\ast$ are formed. Note that $\mathcal{A}^-_\varsigma+ \mathcal{A}^+_\varsigma\approx \mathcal{A}_\varsigma,\quad
(\mathcal{A}^-_{\varsigma^\ast})^\ast+ (\mathcal{A}^+_{\varsigma^\ast})^\ast \approx (\mathcal{A}_{\varsigma^\ast})^\ast,$
$\mathcal{B}^-_\varsigma+
\mathcal{B}^+_\varsigma\approx \mathcal{B}_\varsigma,\quad (\mathcal{B}^-_{\sigma^{-1}})^\ast+ (\mathcal{B}^+_{\sigma^{-1}})^\ast \approx
(\mathcal{B}_{\sigma^{-1}})^\ast$
for all $p>1$ and $q>0$, while $(\mathcal{B}^-_{\rho})^\ast+ (\mathcal{B}^+_{\rho})^\ast \approx (\mathcal{B}_{\rho})^\ast$ if $p>1$ and
$q>1$ only.

The main result of this section is stated in the following theorem.
\begin{theorem}\label{t1} Let $p>1$, $q>0$, $q\not=1$ 
and the operator $\mathcal{H}$ be defined by \eqref{1} with $a(x),b(x)$ satisfying the conditions \eqref{3}. Suppose that $\sigma(x)$ is a
strictly increasing continuous function on $(0,\infty)$ such that $a(x)< \sigma(x)< b(x)$ for $x>0$, and $\rho(y)$ is a strictly increasing
continuous function on $(0,\infty)$ such that $b^{-1}(y)< \rho(y)< a^{-1}(y)$ for $y>0$. Let $\varsigma(x)$ denote either $\sigma(x)$ or
$\rho^{-1}(x)$ on $(0,\infty)$.\begin{enumerate}
\item[{\rm (a)}] If $1<p\le q<\infty$ then $\|\mathcal{H}\|_{L_v^{p}\to L_w^{q}}\ll\mathcal{A}_\varsigma$ and $\|\mathcal{H}\|_{L_v^{p}\to L_w^{q}}
\ll(\mathcal{A}_{\varsigma^\ast})^\ast$.\item[{\rm (b)}] Let $p>1$ and $0<q<p<\infty$. Then $\|\mathcal{H}\|_{L_v^{p}\to L_w^{q}}\ll\mathcal{B}_\varsigma$. If $1<q<p<\infty$ then
$\|\mathcal{H}\|_{L_v^{p}\to 
L_w^{q}}\ll(\mathcal{B}_{\varsigma^\ast})^\ast$.\item[{\rm (c)}]  Let $0<q<1<p<\infty$. If $\rho$ is the dual fairway--function satisfying \eqref{10} then $\|\mathcal{H}\|_{L_v^{p}\to L_w^{q}}
\ll[(\mathcal{B}_\rho^-)^\ast+(\mathcal{B}_\rho^+)^\ast]$. If $\sigma$ is the fairway satisfying \eqref{10v} then $\|\mathcal{H}\|_{L_v^{p}
\to L_w^{q}}\ll (\mathcal{B}_{\sigma^{-1}})^\ast$.\item[{\rm (d)}] Let $1<p\le q<\infty$. If $\sigma$ is the fairway--function satisfying \eqref{10v} then $\|\mathcal{H}\|_{L_v^{p}\to L_w^{q}}\gg
\mathcal{A}_\sigma$ and $\|\mathcal{H}\|_{L_v^{p}\to L_w^{q}}\gg(\mathcal{A}_{\sigma^{-1}})^\ast.$
If $\rho$ is the dual fairway--function satisfying \eqref{10} then $\|\mathcal{H}\|_{L_v^{p}\to L_w^{q}}\gg\mathcal{A}_{\rho^{-1}}$ and
$\|\mathcal{H}\|_{L_v^{p}\to L_w^{q}}\gg(\mathcal{A}_\rho)^\ast.$\item[{\rm (e)}] Let $0<q<p<\infty$, $p>1$. If $\sigma$ is the fairway--function satisfying \eqref{10v} then  $\|\mathcal{H}\|_{L_v^{p}\to L_w^{q}}
\gg\mathcal{B}_\sigma$ and $\|\mathcal{H}\|_{L_v^{p}\to L_w^{q}}\gg[(\mathcal{B}_{\sigma^{-1}}^-)^\ast +(\mathcal{B}_{\sigma^{-1}}^+)^\ast]
\approx (\mathcal{B}_{\sigma^{-1}})^\ast.$ If $\rho$ is the dual fairway--function satisfying \eqref{10} then $\|\mathcal{H}\|_{L_v^{p}\to
L_w^{q}}\gg\mathcal{B}_{\rho^{-1}}$ and $\|\mathcal{H}\|_{L_v^{p}\to L_w^{q}}\gg[(\mathcal{B}_\rho^-)^\ast+(\mathcal{B}_\rho^+)^\ast]\ge
(\mathcal{B}_\rho)^\ast.$\end{enumerate}  \end{theorem}
\begin{proof} (a) To prove the inequalities $\|\mathcal{H}\|_{L_v^{p}\to L_w^{q}}\ll\mathcal{A}_\varsigma$ and $\|\mathcal{H}\|_{L_v^{p}\to
L_w^{q}}\ll\mathcal{A}_\varsigma^\ast$ we refer to the estimate
\begin{equation}\label{u13}\|\mathcal{H}\|_{L_v^p\to L_w^q}
\approx\sup_{t>0}\sup_{b^{-1}(a(t))\le s\le t}\mathbb{A}(s,t)\end{equation} from \cite[Th. 3.2]{SUmia} (see also \cite[Th. 3]{SU} and
\cite[Th. 2.2]{HS}) with $$\mathbb{A}(s,t):= \biggl(\int_s^t
w(x)\,\mathrm{d}x\biggr)^{\frac{1}{q}}
\biggl(\int_{a(t)}^{b(s)}\varrho(y)\,\mathrm{d}y\biggr)^{\frac{1}{p'}}.$$ Put $\tau:=\varsigma^{-1}(a(t))$. We can write that
\begin{equation*}\sup_{b^{-1}(a(t))\le s\le t}\mathbb{A}(s,t)
\le \sup_{b^{-1}(a(t))\le s\le
\tau<t}\mathbb{A}(s,t)+\sup_{\tau\le s\le
t}\mathbb{A}(s,t)=:H_1(t)+H_2(t),\end{equation*} where
\begin{equation*}H_1(t)=\sup_{b^{-1}(\varsigma(\tau))\le
s\le\tau} \biggl(\int_s^{a^{-1}(\varsigma(\tau))}
w(x)\,\mathrm{d}x\biggr)^{\frac{1}{q}}
\biggl(\int_{\varsigma(\tau)}^{b(s)}\varrho(y)\,\mathrm{d}y\biggr)^{\frac{1}{p'}} \end{equation*} and \begin{equation*}H_2(t)\le
\sup_{\tau\le s\le t}
\bigl[W_{\delta^+}(s)\bigr]^{\frac{1}{q}}
\bigl[V_{\Delta}(s)\bigr]^{\frac{1}{p'}},\end{equation*} since $t\le a^{-1}(\varsigma(s)))$, $a(s)\le a(t)$ if $\tau\le s\le t$. Then
\begin{equation*}H_1(t)+H_2(t)\le \bigl[W_{\delta}(\tau)\bigr]^{\frac{1}{q}}
\bigl[V_{\Delta^+}(\tau)\bigr]^{\frac{1}{p'}}+\sup_{s>0} \bigl[W_{\delta^+}(s)\bigr]^{\frac{1}{q}}
\bigl[V_{\Delta}(s)\bigr]^{\frac{1}{p'}}
\ll \mathcal{A}_\varsigma,\end{equation*} and the estimate $\|\mathcal{H}\|_{L_v^{p}\to L_w^{q}}\ll\mathcal{A}_\varsigma$ is proved. Moreover,
it yields, by duality, that $\|\mathcal{H}\|_{L_v^{p}\to L_w^{q}}=
\|\mathcal{H}^\ast\|_{L_{w^{1-q'}}^{q'}\to L_{v^{1-p'}}^{p'}}\ll
(\mathcal{A}_{\varsigma^\ast})^\ast$ where \begin{equation*}
\mathcal{H}^\ast f(y)=\int_{b^{-1}(y)}^{a^{-1}(y)}f(x)\,\mathrm{d}x.\end{equation*}

(b) To prove the estimate $\|\mathcal{H}\|_{L_v^{p}\to L_w^{q}}\ll\mathcal{B}_\varsigma$ in the case $0<q<p<\infty$, $p>1$ we define a
sequence $\{\xi_k\}_{k\in\mathbb{Z}}\subset(0,\infty)$ so that
\begin{equation}\label{seq}\xi_0=1,\hspace{1cm}\xi_k=(a^{-1}\circ
b)^k(1),\hspace{1cm}k\in{\mathbb{Z}},\end{equation}
and put $\eta_k:=a(\xi_k)=b(\xi_{k-1}),$ $\Delta_k:=[\xi_k,\xi_{k+1}),$ $\delta_k:=[\eta_k,\eta_{k+1})$. Breaking
the semiaxis $(0,\infty)$ by the points $\{\xi_k\}_{k\in\mathbb{Z}}$ we
decompose the operator $\mathcal{H}$ into the sum
\begin{equation} \label{29}
\mathcal{H}=T+S
\end{equation} of two block--diagonal operators
$T=\sum_{k\in{\mathbb{Z}}}
T_k $ and $S=\sum_{k\in{\mathbb{Z}}} S_k,$ where
\begin{align*}T_kg(x)=&\int_{a(x)}^{a(\xi_{k+1})}g(y)\,\mathrm{d}y, &
T_k&:L_v^{p}(\delta_k)\to L_w^{q}(\Delta_{k});\\
S_kg(x)=&\int_{b(\xi_k)}^{b(x)}g(y)\,\mathrm{d}y, &
S_k&:L_v^{p}(\delta_{k+1})\to L_w^{q}(\Delta_{k})\end{align*}
and $\sqcup_{k\in\mathbb{Z}}\Delta_k=(0,\infty)$, $\sqcup_{k\in\mathbb{Z}}\delta_k=(0,\infty)$.
Taking into account
the point $\tau_k:=\varsigma^{-1}(b(\xi_k))=\varsigma^{-1}(a(\xi_{k+1}))$ we write
\begin{equation}\label{iSk}S_kg(x)=S_k^-g(x)+S_{k}^+g(x),\end{equation} where \begin{align*}
S_k^-g(x)=&\,\chi_{(\xi_k,\tau_k)}(x)S_kg(x), & S_k^-&:L_v^p(b(\xi_{k}),b(\tau_k))\to L_w^q(\xi_k,\tau_k);\\
S_k^+g(x)=&\,\chi_{(\tau_k,\xi_{k+1})}(x)S_kg(x), & S_k^+&:L_v^p(\delta_{k+1})\to L_w^q(\tau_k,\xi_{k+1}).\end{align*}
Applying the estimate (2.44) from \cite[Lem. 2.3]{SUmia} to the norm of the operator $S_k^-$ we obtain, using~$q-r/p+qr/p=r/p'$, that
\begin{align*}\|S_{k}^-\|
^r\approx&\,
\int_{\xi_k}^{\tau_k}\biggl(\int_{\xi_k}^t\biggl[\int_{b(\xi_k)}^{b(x)}\varrho(y) \,\mathrm{d}y\biggr]^q
w(x)\,\mathrm{d}x\biggr)^{\frac{r}{p}}
\biggl[\int_{b(\xi_k)}^{b(t)}\varrho(y)\,\mathrm{d}y\biggr]^{q-\frac{r}{p}}w(t) \,\mathrm{d}t\\\le&\,
\int_{\xi_k}^{\tau_k}\biggl(\int_{\xi_k}^t
w(x)\,\mathrm{d}x\biggr)^{\frac{r}{p}}
\biggl[\int_{b(\xi_k)}^{b(t)}\varrho(y)\,\mathrm{d}y\biggr]^{\frac{r}{p'}}w(t) \,\mathrm{d}t.\end{align*} Since $\xi_{k}\le t\le\varsigma^{-1}
(b(\xi_k))$ then $b^{-1}(\varsigma(t))\le\xi_k$ and $\varsigma(t)\le b(\xi_k).$
Therefore, \begin{equation}\label{35}
\|S_{k}^-\|^r\ll\int_{\xi_k}^{\tau_k}\bigl[W_{\delta^-}(t)\bigr]^{\frac{r}{p}}
\bigl[V_{\Delta^+}(t)\bigr]^{\frac{r}{p'}}
w(t)\,\mathrm{d}t.\end{equation}
Analogously, on the strength of (2.42) from \cite[Lem. 2.3]{SUmia}, we obtain that
\begin{align}\label{44}\|S_{k}^+\|^r\approx&\,
\int_{\tau_k}^{\xi_{k+1}}\biggl(\int_{t}^{\xi_{k+1}}
w(x)\,\mathrm{d}x\biggr)^{\frac{r}{p}}
\biggl(\int_{b(\xi_k)}^{b(t)}
\varrho(y)\,\mathrm{d}y\biggr)^{\frac{r}{p'}}
w(t)\,\mathrm{d}t\nonumber\\\le&\,
\int_{\tau_k}^{\xi_{k+1}}\bigl[W_{\delta^+}(t)\bigr]^{\frac{r}{p}}
\bigl[V_{\Delta}(t)\bigr]^{\frac{r}{p'}}
w(t)\,\mathrm{d}t,\end{align} since $\varsigma^{-1}(a(\xi_{k+1}))\le t \le\xi_{k+1}$ and, therefore, $\xi_{k+1}\le a^{-1}(\varsigma(t))$ and
$a(t)\le a(\xi_{k+1})=b(\xi_k)$. Thus, by \eqref{iSk}---\eqref{44} it
holds that
\begin{equation}\label{(i)}\|S_k\|
^r\ll \int_{\xi_k}^{\tau_{k}}\bigl[W_{\delta^-}\bigr]^{\frac{r}{p}}
\bigl[V_{\Delta^+}\bigr]^{\frac{r}{p'}}
w
+\int_{\tau_k}^{\xi_{k+1}}\bigl[W_{\delta^+}\bigr]^{\frac{r}{p}}
\bigl[V_{\Delta}\bigr]^{\frac{r}{p'}}
w
.\end{equation} To estimate $\|T_k\|$ we make a decomposition
\begin{equation}\label{Tkw}
T_{k}g(x)=\chi_{(\xi_k,\tau_k)}(x)T_{k}g(x)+\chi_{(\tau_k,\xi_{k+1})}(x)
T_kg(x)=:T_{k}^-g(x)+T_{k}^+g(x).\end{equation} By (2.50) from \cite[Lem. 2.4]{SUmia} we obtain that
\begin{equation*}\|T_{k}^-\|
^r\approx\int_{\xi_k}^{\tau_k}\biggl(\int_{\xi_k}^{t}
w(x)\,\mathrm{d}x\biggr)^{\frac{r}{p}}\biggl(
\int_{a(t)}^{b(\xi_k)}
\varrho(y)\,\mathrm{d}y \biggr)^{\frac{r}{p'}}w(t)\,\mathrm{d}t.\end{equation*} Since $\xi_k\le t\le\varsigma^{-1}(b(\xi_k))$ then
$b^{-1}(\varsigma(t))\le\xi_k$, $b(\xi_k)\le b(t)$ and, therefore, \begin{equation}\label{57}\|T_{k}^-\|
^r\ll \int_{\xi_k}^{\tau_k}
\bigl[W_{\delta^-}(t)\bigr]^{\frac{r}{p}}
\bigl[V_{\Delta}(t)\bigr]^{\frac{r}{p'}}
w(t)\,\mathrm{d}t.\end{equation} Further, by (2.52) from \cite[Lem. 2.4]{SUmia} and in view of $q-r/p+qr/p=r/p'$, \begin{align}\label{58}\|T_{k}^+\|^r
\approx&\,\int_{\tau_k}^{\xi_{k+1}}\biggl(\int_t^{\xi_{k+1}}\biggl[\int_{a(x)}^{b(\xi_k)}
\varrho 
\biggr]^qw(x)\,\mathrm{d}x\biggr)^{\frac{r}{p}}
\biggl(\int_{a(t)}^{b(\xi_k)}\varrho (y)\,\mathrm{d}y
\biggr)^{q-\frac{r}{p}}
w(t)\,\mathrm{d}t\nonumber\\
\le&\,\int_{\tau_k}^{\xi_{k+1}}\biggl(\int_t^{\xi_{k+1}}w(x)\, \mathrm{d}x\biggr)^{\frac{r}{p}}
\biggl(\int_{a(t)}^{b(\xi_k)}\varrho(y)\,\mathrm{d}y\biggr)^{\frac{r}{p'}}
w(t)\,\mathrm{d}t\nonumber\\
\le&\,\int_{\tau_k}^{\xi_{k+1}}
\bigl[W_{\delta^+}(t)\bigr]^{\frac{r}{p}}
\bigl[V_{\Delta^-}(t)\bigr]^{\frac{r}{p'}}
w(t)\,\mathrm{d}t\end{align} since $\varsigma^{-1}(a(\xi_{k+1}))\le t\le\xi_{k+1}$ and, therefore, $\xi_{k+1}\le a^{-1}(\varsigma(t))$,
$b(\xi_k)=a(\xi_{k+1})\le\varsigma(t)$. It follows from
\eqref{Tkw}---\eqref{58} that
\begin{equation*}
\|T_k\|
^r\ll\int_{\xi_k}^{\tau_k}
\bigl[W_{\delta^-}\bigr]^{\frac{r}{p}}
\bigl[V_{\Delta}\bigr]^{\frac{r}{p'}}
w
+\int_{\tau_k}^{\xi_{k+1}}
\bigl[W_{\delta^+}\bigr]^{\frac{r}{p}}
\bigl[V_{\Delta^-}\bigr]^{\frac{r}{p'}}
w
.\end{equation*} From here, \eqref{(i)} and \eqref{29} we find by \cite[Lem. 1]{SU4} (see also \cite[Lem. 3.1]{SUmia} or \cite[Lem. 2.1]{U1}) that
\begin{multline*}\|\mathcal{H}\|^r\approx
\|T\|^r+\|S\|^r\approx \sum_k\Bigl[\|S_k\|^r +\|T_k\|
^r\Bigr]\ll\sum_k\biggl[
\int_{\xi_k}^{\tau_{k}}\bigl[W_{\delta^-}\bigr]^{\frac{r}{p}}
\bigl[V_{\Delta}\bigr]^{\frac{r}{p'}}
w
\\+\int_{\tau_k}^{\xi_{k+1}}\bigl[W_{\delta^+}\bigr]^{\frac{r}{p}}
\bigl[V_{\Delta}\bigr]^{\frac{r}{p'}}
w
\biggr]\le\sum_k
\int_{\xi_k}^{\xi_{k+1}}\bigl[W_{\delta}\bigr]^{\frac{r}{p}}
\bigl[V_{\Delta}\bigr]^{\frac{r}{p'}}
w
=\mathcal{B}_\varsigma^r.\end{multline*} Now, the required estimate $\|\mathcal{H}\|_{L_v^{p}\to L_w^{q}}\ll\mathcal{B}_\varsigma$ is
proven. The inequality $$\|\mathcal{H}\|_{L_v^{p}\to L_w^{q}}=
\|\mathcal{H}^\ast\|_{L_{w^{1-q'}}^{q'}\to L_{v^{1-p'}}^{p'}}\ll
(\mathcal{B}_{\varsigma^\ast})^\ast$$ follows by duality for $q>1
$.

(c) Assume that $0<q<1<p<\infty$ and introduce the sequence $\{\xi_k\}_{k\in\mathbb{Z}}\subset(0,\infty)$ by \eqref{seq}. Let $\varsigma$
denote either the dual fairway--function $\rho^{-1}$ satisfying \eqref{10} or the fairway--function $\sigma$ satisfying \eqref{10v}. We put
$$\xi_k^-:=\varsigma^{-1}(a(\xi_k)),\quad\xi_k^+:=\varsigma^{-1}(b(\xi_k)),\quad \Delta_k^-:=[\xi_k^-,\xi_k),\quad\Delta_k^+:=[\xi_k,\xi_k^+)$$
and decompose the operator into the sum \begin{equation}\label{u2}\mathcal{H}=H^-+H^+\end{equation} of operators $H^\pm=\sum_{k\in\mathbb{Z}}
H^\pm_k$, where
$$H_k^\pm g(x)=\chi_{\Delta_k^\pm}(x)\int_{a(x)}^{b(x)}g(y)\,\mathrm{d}y$$ and $H_k^-\colon L_v^p[a(\xi_k^-),b(\xi_k))\to L_w^q(\Delta_k^-),$
$H_k^+\colon L_v^p[a(\xi_k),b(\xi_k^+))\to L_w^q(\Delta_k^+).$

Note that
if $\varsigma=\sigma$ then $(\mathcal{B}_{\sigma^{-1}}^-)^\ast+ (\mathcal{B}_{\sigma^{-1}}^+)^\ast\approx (\mathcal{B}_{\sigma^{-1}})^\ast$
because of \eqref{10v}. Thus, the two required estimates on $\|\mathcal{H}\|_{L_v^{p}\to L_w^{q}}$ will be established in this part of the
theorem if we prove the inequality  $\|\mathcal{H}\|_{L_v^{p}\to L_w^{q}}\ll[(\mathcal{B}_{\varsigma^{\ast}}^-)^\ast+
(\mathcal{B}_{\varsigma^{\ast}}^+)^\ast]$.

Suppose $(\mathcal{B}_{\varsigma^{\ast}}^\pm)^\ast<\infty$. Since $q'<0$ then, for simplicity, we 
assume that \begin{equation*}
W_{\Theta}(t)<\infty\quad\textrm{and}\quad 0<V_{\vartheta^\pm}(t)<\infty\quad\textrm{for any}\quad t>0.\end{equation*}  Note that
$$(H_k^-|g|)(x)\le\int_{a(\xi_k^-)}^{b(x)}|g(y)|\,\mathrm{d}y,\qquad
(H_k^+|g|)(x)\le\int_{a(x)}^{b(\xi_k^+)}|g(y)|\,\mathrm{d}y$$
and, therefore, by \cite[Lems. 2.3, 2.4]{SUmia},\begin{gather*}
\|H_k^-\|
^r\ll
\int_{\Delta_k^-}\biggl(\int_z^{\xi_k}w(x)\,\mathrm{d}x\biggr)^{\frac{r}{p}}
\biggl(\int_{a(\xi_k^-)}^{b(z)}\varrho(y)\,\mathrm{d}y\biggr)^{\frac{r}{p'}}w(z)\, \mathrm{d}z,\\
\|H_k^+\|^r\ll
\int_{\Delta_k^+}\biggl(\int_{\xi_k}^zw(x)\,\mathrm{d}x\biggr)^{\frac{r}{p}}
\biggl(\int_{a(z)}^{b(\xi_k^+)}\varrho(y)\,\mathrm{d}y\biggr)^{\frac{r}{p'}}w(z)\, \mathrm{d}z.\end{gather*} We have, by $r/p'=r/q'+1$ and in view of $q'<0$, that
\begin{align}\label{u1}
\|H_k^-\|
^r&\ll\int_{a(\xi_k^-)}^{b(\xi_k)} \biggl(\int_{\max\{\xi_k^-,b^{-1}(t)\}} ^{\xi_k}\biggl[\int_z^{\xi_k}w
\biggr]^{\frac{r}{p}}
\biggl[\int_{a(\xi_k^-)}^{b(z)}\varrho 
\biggr]^{\frac{r}{q'}} w(z)\, \mathrm{d}z\biggr) \varrho(t)\,\mathrm{d}t \nonumber\\
&\le\int_{a(\xi_k^-)}^{b(\xi_k)} \biggl(\int_{\max\{\xi_k^-,b^{-1}(t)\}} ^{\xi_k} w
\biggr)^{\frac{r}{q}}
\biggl(\int_{a(\xi_k^-)}^{\max\{b(\xi_k^-),t\}}\varrho 
\biggr)^{\frac{r}{q'}} \varrho(t)\,\mathrm{d}t\nonumber\\
&=\biggl[\int_{a(\xi_k^-)}^{a(\xi_k)}+\int_{a(\xi_k)}^{b(\xi_k)} \biggr]\biggl(\int_{\max\{\xi_k^-,b^{-1}(t)\}} ^{\xi_k} w
\biggr)^{\frac{r}{q}}
\biggl(\int_{a(\xi_k^-)}^{\max\{b(\xi_k^-),t\}}\varrho 
\biggr)^{\frac{r}{q'}} \varrho(t)\,\mathrm{d}t\nonumber\\
&=:I_k^-+II_k^-.\end{align} Note
that\begin{equation*}
I_k^-=\int_{a(\xi_k^-)}^{a(\xi_k)}\varrho(t)\,\mathrm{d}t \biggl(\int_{\xi_k^-} ^{\xi_k} w(x)\,\mathrm{d}x\biggr)^{\frac{r}{q}}
\biggl(\int_{a(\xi_k^-)}^{b(\xi_k^-)}\varrho(y)\,\mathrm{d}y\biggr)^{\frac{r}{q'}}\end{equation*} and
\begin{equation*}
II_k^-\le\int_{a(\xi_k)}^{b(\xi_k)}\biggl(\int_{b^{-1}(t)} ^{\xi_k} w(x)\,\mathrm{d}x\biggr)^{\frac{r}{q}}
\biggl(\int_{a(\xi_k^-)}^{t}\varrho(y)\,\mathrm{d}y\biggr)^{\frac{r}{q'}} \varrho(t)\,\mathrm{d}t,\end{equation*} because $q'<0$. Since
$a(\xi_k^-)\le a(\varsigma^{-1}(t))$ and $\xi_k\le a^{-1}(t)$ if $a(\xi_k)\le t\le b(\xi_k)$,
\begin{equation}\label{u12}
II_k^-
\le \int_{a(\xi_k)}^{b(\xi_k)}\bigl[W_{\Theta}(t)\bigr]^{\frac{r}{q}}
\bigl[V_{\vartheta^-}(t)\bigr]^{\frac{r}{q'}} \varrho(t)\,\mathrm{d}t.\end{equation}
If $\varsigma=\sigma$ then, by \eqref{10v} and in view of $a(\xi_k)=\sigma(\xi_k^-)$,
\begin{equation*}
I_k^-=\int_{\sigma(\xi_k^-)}^{b(\xi_k^-)}\varrho(t)\,\mathrm{d}t \biggl(\int_{\xi_k^-} ^{\xi_k} w(x)\,\mathrm{d}x\biggr)^{\frac{r}{q}}
\biggl(\int_{a(\xi_k^-)}^{b(\xi_k^-)}\varrho(y)\,\mathrm{d}y\biggr)^{\frac{r}{q'}}; \end{equation*} and, since $a(\xi_k^-)\le a(\sigma^{-1}(t))$
for $\sigma(\xi_k^-)\le t\le b(\xi_k^-)$ and $[\xi_k^-,\xi_k)\subseteq\Theta(t)$, 
\begin{equation}
I_k^-\le\int_{a(\xi_k)}^{b(\xi_k^-)}\bigl[W_{\Theta}(t)\bigr]^{\frac{r}{q}} \bigl[V_{\vartheta^-}(t)\bigr]^{\frac{r}{q'}}\varrho(t)\,\mathrm{d}t.
\end{equation}
If $\varsigma=\rho^{-1}$ then, by \eqref{10},
\begin{equation*}
I_k^-=\int_{a(\xi_k^-)}^{a(\xi_k)}\varrho(t)\,\mathrm{d}t\biggl(\int_{\xi_{k-1}} ^{\xi_k^-} w(x)\,\mathrm{d}x\biggr)^{\frac{r}{q}}
\biggl(\int_{a(\xi_k^-)}^{b(\xi_k^-)}\varrho(y)\,\mathrm{d}y\biggr)^{\frac{r}{q'}}; \end{equation*} and,
since $b(\rho(t))\le b(\xi_k^-)$ if $a(\xi_k^-)\le t\le a(\xi_k)$ and $[\xi_{k-1},\xi_k^-)\subseteq\Theta(t)$, then
\begin{equation}\label{u3}
I_k^-
\le \int_{a(\xi_k^-)}^{a(\xi_k)}\bigl[W_{\Theta}(t)\bigr]^{\frac{r}{q}}
\bigl[V_{\vartheta^+}(t)\bigr]^{\frac{r}{q'}} \varrho(t)\,\mathrm{d}t.\end{equation}
From \eqref{u1}---\eqref{u3}, on the strength of \cite[Lem. 1]{SU4} (see also \cite[Lem. 3.1]{SUmia} or \cite[Lem. 2.1]{U1}), we obtain that \begin{multline} \label{u4}\|H^-\|^r\approx
\sum_k\Bigl[\|H^-_{2k-1}\|
^r+\|H^-_{2k}\|
^r\Bigr]\ll\sum_{k}\int_{a(\xi_k)}^{b(\xi_k)}\bigl[W_{\Theta}\bigr]^{\frac{r}{q}}
\bigl[V_{\vartheta^-}\bigr]^{\frac{r}{q'}} \varrho 
\\+\sum_k\int_{a(\xi_k^-)}^{a(\xi_k)}\bigl[W_{\Theta}\bigr]^{\frac{r}{q}}
\bigl[V_{\vartheta^+}\bigr]^{\frac{r}{q'}} \varrho 
\le[(\mathcal{B}_{\varsigma^{\ast}}^-)^\ast]^r +[(\mathcal{B}_{\varsigma^{\ast}}^+)^\ast]^r.\end{multline}
Analogously to \eqref{u1}---\eqref{u3}, we can find that
\begin{equation*}\|H_k^+\|^r
\ll \int_{a(\xi_k)}^{b(\xi_k)}\bigl[W_{\Theta}\bigr]^{\frac{r}{q}}
\bigl[V_{\vartheta^+}\bigr]^{\frac{r}{q'}} \varrho 
+\int_{b(\xi_k)}^{b(\xi_k^+)}\bigl[W_{\Theta}\bigr]^{\frac{r}{q}}
\bigl[V_{\vartheta^-}\bigr]^{\frac{r}{q'}} \varrho 
\end{equation*} and, therefore, \begin{multline}\label{u5}\|H^+\|^r\ll\sum_{k}\int_{a(\xi_k)}^{b(\xi_k)}\bigl[W_{\Theta}\bigr]^{\frac{r}{q}}
\bigl[V_{\vartheta^+}\bigr]^{\frac{r}{q'}} \varrho 
+\sum_k\int_{b(\xi_k)}^{b(\xi_k^+)}\bigl[W_{\Theta}\bigr]^{\frac{r}{q}}
\bigl[V_{\vartheta^-}\bigr]^{\frac{r}{q'}} \varrho 
\\\le[(\mathcal{B}_{\varsigma^{\ast}}^-)^\ast]^r+[(\mathcal{B}_{\varsigma^{\ast}}^+)^\ast]^r.\end{multline} Then $\|\mathcal{H}\|_{L_v^{p}\to L_w^{q}}
\ll[(\mathcal{B}_{\varsigma^{\ast}}^-)^\ast+(\mathcal{B}_{\varsigma^{\ast}}^+)^\ast]$ by \eqref{u2}, \eqref{u4} and \eqref{u5}.

(d) Let $\varsigma$ denote either the fairway--function $\sigma$ satisfying \eqref{10v} or the dual fairway--function $\rho^{-1}$ satisfying
\eqref{10}. 

To prove the required lower estimates in the case $1<p\le q<\infty$ we refer to \eqref{u13} again. We have for any $t>0$:
\begin{multline*}\bigl[W_{\delta^-}(t)\bigr]^{\frac{1}{q}} \bigl[V_{\Delta^-}(t)\bigr]^{\frac{1}{p'}}+\bigl[W_{\delta^+}(t)\bigr]^{\frac{1}{q}}
\bigl[V_{\Delta^+}(t)\bigr]^{\frac{1}{p'}}\\
=\mathbb{A}(b^{-1}(\varsigma(t)),t)+\mathbb{A}(t,a^{-1}(\varsigma(t)))\ll
\sup_{t>0}\sup_{b^{-1}(a(t))\le s\le t}
\mathbb{A}(s,t).\end{multline*} Thus and \eqref{u13}, if $\varsigma=\sigma\in\eqref{10v}$ then $\|\mathcal{H}\|_{L_v^{p}\to L_w^{q}}\gg\mathcal{A}_\sigma$; and, by duality, if $\varsigma^\ast=\rho\in\eqref{10}$ then $\|\mathcal{H}\|_{L_v^{p}\to L_w^{q}}=\|\mathcal{H}^\ast\|_{L_{w^{1-q'}}^{q'}\to L_{v^{1-p'}}^{p'}}\gg(\mathcal{A}_\rho)^\ast$. The substitutions $\sigma(t)=\tau$ and $\rho(t)=\tau$ yield, respectively, the estimate $\|\mathcal{H}\|_{L_v^{p}\to L_w^{q}}\gg(\mathcal{A}_{\sigma^{-1}})^\ast$ from $\|\mathcal{H}\|_{L_v^{p}\to L_w^{q}}\gg\mathcal{A}_\sigma$ if $\varsigma^{\ast}=\sigma^{-1}$, and $\|\mathcal{H}\|_{L_v^{p}\to L_w^{q}}\gg\mathcal{A}_{\rho^{-1}}$ from
$\|\mathcal{H}\|_{L_v^{p}\to L_w^{q}}\gg(\mathcal{A}_\rho)^\ast$ for the case
$\varsigma=\rho^{-1}$.

(e) It is known from \cite[Th. 4.1]{SUmia} that $\|\mathcal{H}\|_{L_v^{p}\to L_w^{q}}\gg \mathcal{B}_\sigma$ if $0<q<p<\infty$, $p>1$ and if
the fairway--function $\sigma$ satisfies the condition \eqref{10v}. The method of the proof of this estimate on $\|\mathcal{H}\|_{L_v^{p}\to L_w^{q}}$ in \cite[Th. 4.1]{SUmia} implies
also the inequality $\|\mathcal{H}\|_{L_v^{p}\to L_w^{q}}\gg (\mathcal{B}_{\sigma^{-1}})^\ast$ for the case when $\sigma\in\eqref{10v}$. Indeed,
suppose first that $\sigma(x)=x$. Then \begin{equation}\label{u17}\|\mathcal{H}\|_{L_v^{p}\to L_w^{q}}\gg\lim_{N\to\infty}\Bigl(\sum_{|k|\le N}
\lambda_k\Bigr)^{\frac{1}{r}}\end{equation} (see \cite[pp.~477--481]{SUmia} for details), where $$\lambda_k=\int_{\eta_{k-1}}^{\eta_{k+2}}
\biggl(\int_{t}^{\eta_{k+2}}w(x)\, \mathrm{d}x\biggr)^{\frac{r}{q}}\biggl(\int_{\eta_{k-1}}^t\varrho(y)\, \mathrm{d}y\biggr)^{\frac{r}{q'}}
\varrho(t)\,\mathrm{d}t$$ and
\begin{equation}\label{u15} \eta_0=1,\quad\eta_{k+1}=a^{-1}(\eta_k),\quad\eta_{k-1}=a(\eta_k),\qquad k\in\mathbb{Z}. \end{equation} If $q'>0$
then, on the strength of $\eqref{10v}\ni\sigma(x)=x$ and, similarly to that in the proof of the estimate $\|\mathcal{H}\|_{L_v^{p}\to L_w^{q}}\gg
\mathcal{B}_\sigma$ in \cite[Th. 4.1]{SUmia},
\begin{equation}\label{u18}\lambda_k\ge
\int_{\eta_{k}}^{\eta_{k+1}}\biggl(\int_{\Theta^+(t)}w(x)\, \mathrm{d}x\biggr)^{\frac{r}{q}}\biggl(\int_{\vartheta(t)}\varrho(y)\, \mathrm{d}y
\biggr)^{\frac{r}{q'}}\varrho(t)\,\mathrm{d}t,\end{equation} since $\Theta^+(t)=[t,a^{-1}(t))\subseteq[t,\eta_{k+2})$ and $\vartheta^-(t)=[a(t),t)
\subseteq[\eta_{k-1},t)$ if $t\in[\eta_{k},\eta_{k+1})$. The same estimate on $\lambda_k$ for $q'<0$ follows by
\begin{multline*}\int_{\eta_{k-1}}^{t}\varrho(y)\, \mathrm{d}y=\int_{\eta_{k-1}}^{\eta_k}\varrho(y)\, \mathrm{d}y +\int_{\eta_{k}}^{t}\varrho(y)\,
\mathrm{d}y=\int_{\eta_{k}}^{b(\eta_k)}\varrho(y)\, \mathrm{d}y
+\int_{\eta_{k}}^{t}\varrho(y)\, \mathrm{d}y\\\le
2\int_{\eta_{k}}^{b(t)}\varrho(y)\, \mathrm{d}y \le
2\int_{a(t)}^{b(t)}\varrho(y)\, \mathrm{d}y,\qquad t\in[\eta_k,\eta_{k+1}). \end{multline*} Combining \eqref{u17} with \eqref{u18} we obtain that
$\|\mathcal{H}\|_{L_v^{p}\to L_w^{q}}\gg (\mathcal{B}_{\sigma^{-1}}^+)^\ast$ in view of
$\sqcup_{k\in\mathbb{Z}}[\eta_k,\eta_{k+1})=(0,\infty)$. The inequality $\|\mathcal{H}\|_{L_v^{p}\to L_w^{q}}\gg (\mathcal{B}_{\sigma^{-1}}^-)^\ast$
can be proved similarly, using the intervals $[\zeta_k,\zeta_{k+1})$ formed by the boundary $b(x)$: \begin{equation}\label{u19} \zeta_0=1,
\quad\zeta_{k+1}=b(\zeta_k),\quad\zeta_{k-1}=b^{-1}(\zeta_k),\qquad k\in\mathbb{Z}. \end{equation} Therefore, we have that $\|\mathcal{H}\|_{L_v^{p}\to L_w^{q}}\gg(\mathcal{B}_{\sigma^{-1}}^-)^\ast+(\mathcal{B}_{\sigma^{-1}}^+)^\ast\approx
(\mathcal{B}_{\sigma^{-1}})^\ast$ if $\sigma(x)=x$. The same statement with general $\sigma\in\eqref{10v}$ follows from the case $\sigma(x)=x$ by
substitutions $\tilde{a}(x)=\sigma^{-1}(a(x))$, $\tilde{b}(x)=\sigma^{-1}b((x))$, $\tilde{f}(y)=f(\sigma(t))\sigma'(t)$ and  $\tilde{v}(x)=v(\sigma(x))(\sigma'(x))^{1-p}$.

Now let $\rho$ be the dual fairway satisfying \eqref{10}.
If $1<q<p<\infty$ then the estimates $\|\mathcal{H}\|_{L_v^{p}\to L_w^{q}}\gg (\mathcal{B}_\rho)^\ast$ and $\|\mathcal{H}\|_{L_v^{p}\to L_w^{q}}\gg
\mathcal{B}_{\rho^{-1}}$ follows, by duality, from $\|\mathcal{H}\|_{L_v^{p}\to L_w^{q}}\gg \mathcal{B}_\sigma$ and $\|\mathcal{H}\|_{L_v^{p}\to
L_w^{q}}\gg(\mathcal{B}_{\sigma^{-1}})^\ast$, respectively 
(see also
\cite[Th. 2.2]{SU3}). Since $r/q'>0$, we also have $(\mathcal{B}_\rho)^\ast\approx(\mathcal{B}_\rho^-)^\ast+(\mathcal{B}_\rho^+)^ \ast.$

It only rests to prove the estimates $\|\mathcal{H}\|_{L_v^{p}\to L_w^{q}}\gg [(\mathcal{B}_\rho^-)^\ast+(\mathcal{B}_\rho^+)^ \ast]$
and \\$\|\mathcal{H}\|_{L_v^{p}\to L_w^{q}}\gg \mathcal{B}_{\rho^{-1}}$ for the case when $0<q<1<p<\infty$ and $\rho\in\eqref{10}$
. To this end we assume that $\|\mathcal{H}\|_{L_v^{p}\to L_w^{q}}<\infty$.
The inequalities $\|\mathcal{H}\|_{L_v^{p}\to L_w^{q}}\gg [(\mathcal{B}_\rho^-)^\ast+(\mathcal{B}_\rho^+)^ \ast]$ and
$\|\mathcal{H}\|_{L_v^{p}\to L_w^{q}}\gg \mathcal{B}_{\rho^{-1}}$ will be established if we show that $\|\mathcal{H}\|_{L_v^{p}\to L_w^{q}}
\gg (\mathcal{B}_\rho^\pm)^\ast$ and $\|\mathcal{H}\|_{L_v^{p}\to L_w^{q}}\gg \mathcal{B}_{\rho^{-1}}^\pm$, where $\mathcal{B}_{\rho^{-1}}\approx \mathcal{B}_{\rho^{-1}}^-+\mathcal{B}_{\rho^{-1}}^+$.

Assume that $\rho(y)=y$ first. In order to prove that
$\|\mathcal{H}\|_{L_v^{p}\to L_w^{q}}\gg (\mathcal{B}_\rho^+)^\ast$ and $\|\mathcal{H}\|_{L_v^{p}\to L_w^{q}}\gg \mathcal{B}_{\rho^{-1}}^-$
we introduce the sequence \eqref{u15} and form intervals $\varkappa_j^k$, $j=1,\ldots,j_b$, on each $[\eta_k,\eta_{k+1})$, $k\in\mathbb{Z}$, as follows: 
\begin{enumerate} 
\item if
$\eta_{k+1}\le b(\eta_k)$ then $j_b=1$ and $[\eta_k,\eta_{k+1})=\varkappa_1^k$; 
\item if
$b(\eta_k)<\eta_{k+1}\le b(b(\eta_k))$ then $j_b=2$ and $[\eta_k,\eta_{k+1})=\cup_{j=1}^2\varkappa_j^k$, where $\varkappa_1^k=[\eta_k,b(\eta_k))$ and $\varkappa_2^k=[b^{-1}(\eta_{k+1}),\eta_{k+1})$;
\item if
$b(b(\eta_k))<\eta_{k+1}\le b(b(b(\eta_k)))$ then $j_b=3$ and $[\eta_k,\eta_{k+1})=\cup_{j=1}^3\varkappa_j^k$, where $\varkappa_1^k=[\eta_k,b(\eta_k))$, $\varkappa_2^k=[b(\eta_k),b(b(\eta_k)))$ and $\varkappa_3^k=[b^{-1}(\eta_{k+1}),\eta_{k+1})$; \item[~] $\ldots$~. \end{enumerate} Finally, we obtain for each $k\in\mathbb{Z}$ that $[\eta_k,\eta_{k+1})=\cup_{j=1}^{j_b}\varkappa_j^k$, where $\varkappa_j^k=[b^{(j-1)}(\eta_k),b^{(j)}(\eta_k))$ for $j=1,\ldots,j_b-1$ and 
$\varkappa_{j_b}^k=[b^{-1}(\eta_{k+1}),\eta_{k+1})$. It holds for any $k\in\mathbb{Z}$ and $j\in\{1,\ldots,j_b\}$ that
\begin{equation}\label{u21}
W_\Theta(y)\approx\int_{\varkappa_j^k}w(x)\,\mathrm{d}x,\qquad y\in\varkappa_j^k\end{equation} (see \cite[Lem. 2.7 with $b^{-1}$, $a^{-1}$ and $\rho$ instead of $\sigma$]{SUmia} for details). Let $$l_{k,j}:=\biggl(\int_{\varkappa_{j}^k}w(x)\,\mathrm{d}x\biggr)^{\frac{r}{pq}} \biggl(\int_{a(m_{j}^k)}^{m_{j}^k}\varrho(y)\,
\mathrm{d}y\biggr)^{\frac{r}{pq'}},$$ where $m_j^k$ is the right end point of the interval $\varkappa_j^k$, and define a function $$g_a(t):=\sum_{k=-N}^{N}\sum_{j=1}^{j_b(k)}\chi_{[a(m_{j}^k),{m_{j}^k}]}(t)
\,l_{k,j} \,\varrho(t),\qquad N\in\mathbb{N}.$$ By the construction, $[a(x),b(x)]\supset 
[a(m_{j}^k),m_{j}^k]$ if $x\in\varkappa_j^k$, and each $x\in(0,\infty)$ has intersection with two $\varkappa_j^k$ at most.
Then, by $r/(pq')+1=r/(p'q)$ and $r/p+1=r/q$, we obtain for any $N\in\mathbb{N}$ that \begin{align}\label{u14}
2\|\mathcal{H}g_a\|^q_{w,q}&\ge \sum_{k=-N}^{N}
\sum_{j=1}^{j_b(k)}\int_{\varkappa_{j}^k}\biggl(\int_{a(x)}^{b(x)}g_a(y)\,\mathrm{d}y \biggr)^qw(x)\,\mathrm{d}x\nonumber\\&\ge\sum_{k=-N}^{N}
\sum_{j=1}^{j_b(k)}\int_{\varkappa_j^k} \biggl(\int_{a(x)}^{b(x)} l_{k,j}\,\varrho(y)\chi_{[a(m_{j}^k),{m_{j}^k}]}(y)\,\mathrm{d}y
\biggr)^qw(x)\,\mathrm{d}x\nonumber\\
&= \sum_{k=-N}^{N}\sum_{j=1}^{j_b(k)}l_{k,j}^q\biggl(\int_{a(m_{j}^k)}^{m_{j}^k}  \varrho(y)\,\mathrm{d}y \biggr)^q\int_{\varkappa_j^k}
w(x)\,\mathrm{d}x\nonumber\\
&=\sum_{k=-N}^{N}\sum_{j=1}^{j_b(k)}\biggl(\int_{\varkappa_j^k}
w
\biggr)^{\frac{r}{q}}\biggl(\int_{a(m_{j}^k)}^{m_{j}^k}
\varrho 
\biggr)^{\frac{r}{p'}}=:\sum_{k=-N}^{N}\sum_{j=0}^{j_b(k)-1}\lambda_{k,j}.
\end{align} 
Since 
the intervals $[m_{j-1}^k,m_j^k)$, which are formed by the right end points $m_j^k$ of $\varkappa_j^k$ with $m_0^k=\eta_k$, are disjoint, then
\begin{align}\label{u22}
\|g_a\|^p_{v,p}=&\sum_{k=-N-1}^{N}\sum_{j=1}^{j_b(k)}\int_{m_{j-1}^k}^{m_{j}^k} |g_a(y)|^pv(y)\,\mathrm{d}y\nonumber\\=&
\sum_{k=-N-1}^{N}\sum_{j=1}^{j_b(k)}\int_{m_{j-1}^k}^{m_{j}^k} \biggl[\sum_{i=j}^{j_b(k)}l_{k,i}^p
\chi_{[a(m_{i}^k),{m_{i}^k}]}(y) \nonumber\\+\sum_{i=1}^{j_b(k+1)}l_{k+1,i}^p&\chi_{[a(m_{i}^{k+1}),{m_{i}^{k+1}}]}(y) \biggr]\varrho(y)\,
\mathrm{d}y
=:\sum_{k=-N-1}^{N}\sum_{j=1}^{j_b(k)}I_{k,j}+II_{k,j}.\end{align}  Denote \begin{align*}\Xi(k)&:=\{i=j,\ldots,j_b(k)
\colon ({m_{j-1}^k},{m_{j}^k})\cap(a(m_{i}^{k}),{m_{i}^{k}})\not=\emptyset\},\\
\Xi(k+1)&:=\{i=1,\ldots,j_b(k+1)\colon ({m_{j-1}^k},{m_{j}^k})\cap(a(m_{i}^{k+1}),{m_{i}^{k+1}})\not=\emptyset\} \end{align*} and
consider $|k|\le N$ first. Notice that for all such $k$ we have $\Xi(k)=\{j,\ldots,j_b\}$  since $a(m_{j_b(k)}^k)=a(m_0^{k+1})=
a(\eta_{k+1})=\eta_k=m_0^k$ by the construction. $\Xi(k+1)=$ is empty if $j_b(k+1)=1$ and $\Xi(k+1)=\{1,\ldots,j_b(k+1)-1\}$ if $j_b(k+1)>1$. 
Therefore, any fixed $|k|\le N$ and $j(k)$,
\begin{equation*}
I_{k,j}+II_{k,j}= \sum_{i=j}^{j_b(k)}l_{k,i}^p \int_{m_{j-1}^k}^{m_{j}^k}\varrho(y)\,\mathrm{d}y +
\sum_{i\in\Xi(k+1)}l_{k+1,i}^p \int_{\max\{m_{j-1}^k,a(m_{i}^{k+1})\}}^{m_{j}^k}\varrho(y)\,\mathrm{d}y.\end{equation*} Since $q'<0$,
then $$l_{k+1,i}\le\biggl(\int_{\varkappa_i^{k+1}}w(x)\,\mathrm{d}x\biggr)^{\frac{r}{pq}} \biggl(
\int_{\max\{m_{j-1}^k,a(m_{i}^{k+1})\}}^{m_{j}^k}\varrho(y)\,\mathrm{d}y\biggr) ^{\frac{r}{pq'}},\quad i\in\Xi(k+1).$$ Moreover, on the strength of \eqref{10} (see also \eqref{u21}),
$$\int_{\eta_{k+1}}^{\eta_{k+2}}w(x)\,\mathrm{d}x= \int_{b^{-1}(\eta_{k+1})}^{\eta_{k+1}}w(x)\,\mathrm{d}x\approx 
\int_{\varkappa_{j_b}^{k}}w(x)\,\mathrm{d}x.$$
Thus, in view of $r/q'+1=r/p'>0$ and $r/q>1$, \begin{align}\label{u23} II_{k,j}&\le\sum_{i\in \Xi(k+1)}\biggl(
\int_{\varkappa_i^{k+1}}w(x)\,\mathrm{d}x\biggr)^{\frac{r}{q}} \biggl(\int_{\max\{m_{j-1}^k,a(m_{i}^{k+1})\}}^{m_{j}^k}\varrho(y)\,
\mathrm{d}y\biggr)^{\frac{r}{p'}}\nonumber\\
&\le\biggl(\int_{a(m_{j_b}^{k})}^{m_{j_b}^k}\varrho(y)\,\mathrm{d}y\biggr)^{\frac{r}{p'}} \sum_{i\in\Xi(k+1)}\biggl(\int_{\varkappa_{i}^{k+1}}w(x)\,\mathrm{d}x\biggr)^{\frac{r}{q}} \nonumber\\
&\le\biggl(\int_{a(m_{j_b}^{k})}^{m_{j_b}^k}\varrho(y)\,\mathrm{d}y\biggr)^{\frac{r}{p'}} \biggl( \sum_{i\in\Xi(k+1)}\int_{\varkappa_i^{k+1}}w(x)\,\mathrm{d}x\biggr) ^{\frac{r}{q}}\nonumber\\&\le 2\biggl(\int_{a(m_{j_b}^{k})}^{m_{j_b}^k}\varrho(y)\,\mathrm{d}y\biggr)^{\frac{r}{p'}}
\biggl(\int_{\eta_{k+1}}^{\eta_{k+2}}w(x)\,\mathrm{d}x\biggr)^{\frac{r}{q}}\approx \lambda_{k,j_b}.
\end{align} Since $q'<0$ and $a(m_{j_b(k)}^k)=m_0^k$,
$$l_{k,i}\le\biggl(\int_{\varkappa_{i}^{k}}w(x)\,\mathrm{d}x\biggr)^{\frac{r}{pq}} \biggl(\int_{m_{j-1}^k}^{m_{j}^k}\varrho(y)\,
\mathrm{d}y\biggr) ^{\frac{r}{pq'}},\qquad i\in\Xi(k).$$ Therefore, by \eqref{10}, analogously to \eqref{u23},
\begin{align}\label{u24} I_{k,j}&\le\biggl(\int_{m_{j-1}^k}^{m_{j}^k}\varrho(y)\,\mathrm{d}y\biggr) ^{\frac{r}{p'}}\ \sum_{i=j}^{j_b(k)}
\biggl(\int_{\varkappa_i^{k}}w(x)\, \mathrm{d}x\biggr)^{\frac{r}{q}} \nonumber\\&\le 2
\biggl(\int_{a(m_{j}^k)}^{m_{j}^k}\varrho (y)\,\mathrm{d}y
\biggr) ^{\frac{r}{p'}}\ \biggl(\int_{m_{j-1}^{k}}^{m_{j_b}^{k}}w(x)\,
\mathrm{d}x
\biggr)^{\frac{r}{q}}\nonumber\\&\le
\biggl(\int_{a(m_{j}^k)}^{m_{j}^k}\varrho (y)\,\mathrm{d}y
\biggr) ^{\frac{r}{p'}}\ \biggl(\int_{m_{j-1}^{k}}^{a^{-1}(m_{j}^{k})}w(x)\,
\mathrm{d}x
\biggr)^{\frac{r}{q}}\nonumber\\
&=2^{\frac{r}{q}}\biggl(\int_{a(m_j^k)}^{m_{j}^k}\varrho (y)\,\mathrm{d}y
\biggr) ^{\frac{r}{p'}}\ \biggl(\int_{m_{j-1}^{k}}^{m_{j}^{k}}w(x)\,
\mathrm{d}x
\biggr)^{\frac{r}{q}}\ll\lambda_{k,j}.
\end{align} If $k=-N-1$ then $\Xi(k)=\emptyset$ by definition of $g_a$.
If $\Xi(k+1)=\Xi(-N)\not=\emptyset$ then, by the same reason as before,
\begin{align}\label{u27} II_{-N-1,j}\le&\sum_{i\in\Xi(-N)}
l_{-N,i}^p \int_{\max\{m_{j-1}^{-N-1},a(m_{i}^{-N})\}}^{m_{1}^{-N}}\varrho(y)\,\mathrm{d}y
\nonumber\\\le&\sum_{i\in\Xi(-N)}
\biggl(\int_{a(m_{i}^{-N})}^{m_{1}^{-N}}\varrho(y)\,\mathrm{d}y\biggr) ^{\frac{r}{p'}} \biggl(\int_{m_{i-1}^{-N}}^{m_{i}^{-N}}w(x)\,
\mathrm{d}x\biggr)^{\frac{r}{q}}\nonumber\\
\le&\biggl(\int_{a(m_{1}^{-N})}^{m_{1}^{-N}}\varrho 
\biggr) ^{\frac{r}{p'}}\sum_{i\in\Xi(-N)}
 \biggl(\int_{m_{i-1}^{-N}}^{m_{i}^{-N}}w 
\biggr)^{\frac{r}{q}}\ll\lambda_{-N,1}.
\end{align}
Combining \eqref{u22}---\eqref{u24} and \eqref{u27} we obtain that
$$\|g_a\|^p_{v,p}\ll\sum_{k=-N}^{N}\sum_{j=1}^{j_b(k)}\lambda_{k,j}.$$
Thus and \eqref{u14} it follows, letting $N\to\infty$, that \begin{equation}\label{u26} \|\mathcal{H}\|_{L_v^p\to L_w^q}\gg
\Bigl(\sum_{k\in\mathbb{Z}}\sum_{j=1}^{j_b(k)}\lambda_{k,j}\Bigr)^{\frac{1}{r}}. \end{equation}

To obtain the estimate $\|\mathcal{H}\|_{L_v^{p}\to L_w^{q}}\gg (\mathcal{B}_\rho^+)^\ast$ from \eqref{u26} notice that  \begin{equation*}
\biggl(\int_{\varkappa_{j}^k}  \varrho(y)\,\mathrm{d}y \biggr)^{\frac{r}{p'}}\approx \int_{\varkappa_{j}^k}
\biggl(\int_t^{m_{j}^k}  \varrho 
\biggr)^{\frac{r}{q'}}\varrho(t)\,\mathrm{d}t
\ge\int_{\varkappa_{j}^k}\biggl(\int_t^{b(t)}  \varrho 
\biggr)^{\frac{r}{q'}}\varrho(t)\,\mathrm{d}t\end{equation*}
by the construction of $\varkappa_j^k$ provided $q'<0$. 
Thus, by \eqref{u21}:
\begin{multline*}\lambda_{k,j}\gg \biggl(\int_{\varkappa_j^{k}}w(x)\, \mathrm{d}x\biggr)^{\frac{r}{q}} \int_{\varkappa_j^k}
\biggl(\int_{\vartheta^+(t)}  \varrho(y)\,\mathrm{d}y \biggr)^{\frac{r}{q'}}\varrho(t)\,\mathrm{d}t\\
\approx \int_{\varkappa_j^k} [W_{\Theta}(t)] ^{\frac{r}{q}}[V_{\vartheta^+}(t)]^{\frac{r}{q'}}\varrho(t)\,\mathrm{d}t \ge \int_{m_{j-1}^k}^{m_j^k} [W_{\Theta}(t)] ^{\frac{r}{q}}[V_{\vartheta^+}(t)]^{\frac{r}{q'}}\varrho(t)\,\mathrm{d}t.
\end{multline*} Now the estimate $\|\mathcal{H}\|_{L_v^p\to L_w^q}$ $\gg (\mathcal{B}_\rho^+)^\ast$ in the case
$\rho(y)=y$ follows from \eqref{u26} in view of $\cup_{k,j}[m_j^k,m_{j+1}^k)=(0,\infty)$.

To prove $\|\mathcal{H}\|_{L_v^{p}\to L_w^{q}}\gg \mathcal{B}_{\rho^{-1}}^-$ if $\rho(y)=y$ notice that, by \eqref{u21}, any two neighbour $\int_{\varkappa_j^k} w$ are equivalent to each other. Therefore, 
the sum, say $\gamma_{k,j}$, 
of two neighbour $\lambda_{k,j}$ 
can be estimated from below as follows: \begin{equation*} \label{u28}
\gamma_{k,j}\gg\biggl(\int_{\varkappa_{j+1}^k} w(x)\,\mathrm{d}x\biggr)^{\frac{r}{q}}\biggl(\int_{a(m_{j}^k)}^{m_{j+1}^k}  \varrho(y)\,
\mathrm{d}y \biggr)^{\frac{r}{p'}}=:\mu_{k,j}\end{equation*} (here $m_{j_b+1}^k=m_1^{k+1}$ and $\varkappa_{j+1}^k=\varkappa_{1}^{k+1}$ if $j=j_b$). 
Therefore, by \eqref{u21}, \begin{multline*}\label{u29}\mu_{k,j}\approx\biggl(
\int_{a(m_{j}^k)}^{m_{j+1}^k}  \varrho(y)\,\mathrm{d}y \biggr)^{\frac{r}{p'}} \int_{\varkappa_{j+1}^{k}}
[W_\delta(t)]^{\frac{r}{p}}w(t)\,\mathrm{d}t\\
\ge\int_{\varkappa_{j+1}^k}
 [W_{\delta}(t)] ^{\frac{r}{p}}[V_{\Delta^-}(t)]^{\frac{r}{p'}}w(t)\,\mathrm{d}t
\ge \int_{m_{j}^k}^{m_{j+1}^k} 
 [W_{\delta}(t)] ^{\frac{r}{p}}[V_{\Delta^-}(t)]^{\frac{r}{p'}}w(t)\,\mathrm{d}t.
\end{multline*} It now follows from \eqref{u26} that
\begin{multline}\|\mathcal{H}\|_{L_v^p\to L_w^q}\gg \Bigl(\sum_{k\in\mathbb{Z}}\sum_{j=1}^{j_b(k)} 2\,\lambda_{k,j}\Bigr)^{\frac{1}{r}}\gg
\Bigl(\sum_{k\in\mathbb{Z}}\sum_{j=1}^{j_b(k)} \mu_{k,j}\Bigr)^{\frac{1}{r}}\\\gg \biggl(\sum_{k\in\mathbb{Z}}\sum_{j=1}^{j_b(k)}
\int_{m_j^k}^{m_{j+1}^k}  [W_{\delta}] ^{\frac{r}{p}}[V_{\Delta^-}]^{\frac{r}{p'}}w
\Bigr)^{\frac{1}{r}}=
\mathcal{B}_{\rho^{-1}}^-. \end{multline} Thus, $\|\mathcal{H}\|_{L_v^p\to L_w^q}\gg \mathcal{B}_{\rho^{-1}}^-$ if $\rho(y)=y$.

Analogously, using the sequence \eqref{u19} instead of \eqref{u15}, one can show that  $\|\mathcal{H}\|_{L_v^p\to L_w^q}\gg
(\mathcal{B}_\rho^-)^\ast$ and $\|\mathcal{H}\|_{L_v^p\to L_w^q}\gg \mathcal{B}_{\rho^{-1}}^+$ if $\rho(y)=y$. The general case of $\rho$ follows
from $\rho(y)=y$ by substitutions $\tilde{a}(x)=a(\rho(x))$, $\tilde{b}(x)=b(\rho(x))$ and $\tilde{w}(x)=w(\rho(x))\rho'(x)$.
\end{proof}

\begin{corollary}
 Let $p>1$, $q>0$, $q\not=1$ 
and the operator $\mathcal{H}$ be defined by \eqref{1} with $a(x),b(x)$ satisfying the conditions \eqref{3}. Suppose that $\sigma(x)$
is the fairway--function on $(0,\infty)$ satisfying \eqref{10v} and $\rho(y)$ is the dual fairway--function on $(0,\infty)$ satisfying
\eqref{10}. Then \begin{equation} \label{pq}\|\mathcal{H}\|_{L_v^{p}\to L_w^{q}}\approx \mathcal{A}_\sigma
\approx(\mathcal{A}_{\sigma^{-1}})^\ast \approx \mathcal{A}_{\rho^{-1}}\approx(\mathcal{A}_{\rho})^\ast\end{equation} if $1<p\le q<\infty$.
If $0<q<p<\infty$, $p>1$ then \begin{equation} \label{qp}\|\mathcal{H}\|_{L_v^{p}\to L_w^{q}}\approx\mathcal{B}_\sigma
\approx(\mathcal{B}_{\sigma^{-1}})^\ast \approx\mathcal{B}_{\rho^{-1}}\approx (\mathcal{B}_\rho^-)^\ast+(\mathcal{B}_\rho^+)^\ast,\end{equation}
where $(\mathcal{B}_\rho^-)^\ast+(\mathcal{B}_\rho^+)^\ast\approx (\mathcal{B}_\rho)^\ast$
if $1<q<p<\infty$.\end{corollary}

\begin{remark} The norm $\|\mathcal{H}\|_{L_v^{p}\to L_w^{q}}$ is characterized by Theorem 4 from \cite[Ch. 7, \S~1.5]{KA}
if $q=1\le p\le\infty$ or $p=1< q\le\infty$. 
Notice also that if the operator $\mathcal{H}:L_v^p\to L_w^q$ is bounded in the case $0<p<1,$ $0<q\le\infty$ then $\|\mathcal{H}\|_{L_v^{p}\to L_w^{q}}=0$ (see \cite[Th.~2]{PS}).

\end{remark}

\section{Fractional inequality}
\label{sec:2}

Let $q>0$ and $p>1$, weights $u$ and $v$ be non--negative on $(0,\infty)$, and a function $f \in \mathcal{W}_{p,v}^1$ satisfy one
of the boundary conditions: $f(0)=0$ or $f(\infty)=0$.
The weighted differential Hardy inequality of the form
\begin{equation}\label{hh}
\left\| f\right\| _{q,u}\leq C\left\| f^{\prime }\right\| _{p,v}
\end{equation} is very well--studied (see \cite{OK} and references there). Higher order
differential inequalities of Hardy type
\begin{equation}\label{M1}
\left\| f\right\| _{q,u}\leq C\left\| f^{(k)}\right\| _{p,v}\qquad(k=1,2,\ldots)
\end{equation} with $f$ such that $f^{(j)}(0)=0$, $j=0,1,\ldots,k-1,$ or $f^{(j)}(\infty)=0$, $j=0,1,\ldots,k-1,$ were studied in [\ref{Ssmz}, \ref{Step}].
An overdetermined analog of \eqref{M1} (i.e. with some additional conditions on $f$) appeared in [\ref{N1}, \ref{NS}, \ref{NS1}]. Some norm inequalities involving fractional derivatives of smooth functions were considered in
the papers \cite{Gr} and \cite{J}. 
For further historical remarks 
we refer to the books \cite{MKP} and \cite{KP}.

This section continues the study of weighted differential inequalities and investigates relations between Lebesgue norms of fractional 
and first derivatives of functions from $\mathcal{W}_{p,v}^1$ on
$(0,\infty)$. As it was mentioned in the introduction, the question arose from the paper \cite{HS}
by H.P. Heinig and G. Sinnamon, where the authors obtained separate necessary and sufficient
conditions for the inequality \eqref{00} to hold. Here we give new conditions for the validity of \eqref{00}
with general weights $u$ and $v$ (see Theorem 3.1). In addition, the case $u=1$ is
separately considered and illustrated by two examples. 
In particular, we state criteria for \eqref{00} to hold with a power weight $v$ in the case $1<p\le q<\infty$ (see Proposition 3.3), and for $v$ with antiderivative $-1/\gamma(1+x^\gamma)$ if $0<q<p<\infty$ (see Proposition 3.7). 

Consider the inequality \eqref{00} in the form
\begin{equation}
   \biggl(\int_0^\infty\!\!\!\! \int_0^\infty |f(x)-f(y)|^q
   \tilde{u}(x,y)\,\mathrm{d}x\,\mathrm{d}y\biggr)^{\frac{1}{q}} \le C \biggl(\int_0^\infty |f'|^p
   v
\biggr)^{\frac{1}{p}}, \label{ineq}
\end{equation} where $\displaystyle\tilde{u}(x,y):=\frac{u(x,y)}{|x-y|^{1+\lambda q}}$ and a constant $0<C<\infty$ is the best possible and independent of $f$. 
Analogously to that established in the proof of Corollary 3.5 from \cite{HS} one can state that \eqref{ineq} is equivalent to the inequality 
\begin{equation*}
   \left(\int_0^1 \left[\int_0^\infty \left|\int_{\xi x}^x f'(z)\,\mathrm{d}z\right|^q
    \tilde{U}_\xi(x)\,\mathrm{d}x\right]\mathrm{d}\xi\right)^{\frac{1}{q}} \le C \left(\int_0^\infty |f'|^p
   v
\right)^{\frac{1}{p}} 
\end{equation*} with $\tilde{U}_\xi(x)=x(\tilde{u}(x,\xi x)+\tilde{u}(\xi x,x))$. Thus, the initial inequality \eqref{00} becomes of the form
\begin{equation}\label{00'}
\left(\int_0^1 \left[\int_0^\infty \left|\int_{\xi x}^x f'(z)\,\mathrm{d}z\right|^q
    U_\xi(x)\,\mathrm{d}x\right]\mathrm{d}\xi\right)^{\frac{1}{q}} \le C \left(\int_0^\infty |f'|^p
   v
\right)^{\frac{1}{p}}
\end{equation} with $$U_\xi(x)=\frac{u(x,\xi x)+u(\xi x,x)}{x^{\lambda q}(1-\xi)^{1+\lambda q}}.$$

Let $u\equiv 1.$ Then $U_\xi(x)=2 x^{-\lambda q}(1-\xi)^{-1-\lambda q}$ and, following the above transformations, \eqref{00} is equivalent to
the inequality
\begin{equation}\label{u=1}
\left(\int_0^1 \left[\int_0^\infty \left|\int_{\xi x}^x f'(z)\,\mathrm{d}z\right|^q
    \frac{\mathrm{d}x}{x^{\lambda q}}\right]\frac{\mathrm{d}\xi} {(1-\xi)^{1+\lambda q}}\right)^{\frac{1}{q}} \le C \left(\int_0^\infty |f'|^p
   v
\right)^{\frac{1}{p}},
\end{equation} involving the Hardy--Steklov integral operator \begin{equation}\label{u10} \mathcal{H}_\xi g(x):=\int_{\xi x}^x g(z)\,\mathrm{d}z
\qquad(0\le \xi\le 1)\end{equation} from $L_v^p$ to $L_w^q$ with $w(x)=x^{-\lambda q}.$ Putting $\|\mathcal{H}_\xi\|:=\|\mathcal{H}_\xi\|_{L_v^p
\to L_w^q}$ we write \begin{multline}\label{first}
C^q=\sup_{\|f'\|\not=0}\frac{\int_0^1 \left[\int_0^\infty \left|\int_{\xi x}^x f'(z)\,\mathrm{d}z\right|^q
    \frac{\mathrm{d}x}{x^{\lambda q}}\right]\frac{\mathrm{d}\xi}{(1-\xi)^{1+\lambda q}}}{\|f'\|_{p,v}^q}\\\le\int_0^1 \sup_{\|f'\|\not=0}
\frac{\int_0^\infty \left|\int_{\xi x}^x f'(z)\,\mathrm{d}z\right|^q
    \frac{\mathrm{d}x}{x^{\lambda q}}}{\|f'\|_{p,v}^q} \frac{\mathrm{d}\xi}
{(1-\xi)^{1+\lambda q}}=
\int_0^1 \frac{\|\mathcal{H}_\xi\|^q \mathrm{d}\xi}
{(1-\xi)^{1+\lambda q}}.\end{multline} On the other hand, 
suppose that $\|\mathcal{H}_\xi\|\approx F(\xi),$ where a functional
\begin{equation*}
F(\xi)=\begin{cases}\mathcal{A}_\varsigma\quad\textrm{or}\quad (\mathcal{A}_{\varsigma^{-1}})^\ast, &1<p\le q<\infty,\\
\mathcal{B}_\varsigma\quad\textrm{or}\quad
(\mathcal{B}_{\varsigma^{-1}}^-)^\ast +(\mathcal{B}_{\varsigma^{-1}}^+)^\ast, &0<q<p<\infty,\ p>1
\end{cases}\end{equation*} is independent of $f$. Assume that a function $f_\xi\in AC(0,\infty)$ is such that $f'_\xi=g_\xi$ a.e., 
$g_\xi\ge 0$ 
and
\begin{equation}\label{second}\!\|\mathcal{H}_\xi\|^q\!=\!\sup_
{\substack{\|g\|\not=0\\g\geq 0}}\!\frac{\int_0^\infty\! \Bigl(\int_{\xi x}^x g(z)\,\mathrm{d}z\!\Bigr)^q\!\frac{\mathrm{d}x}{x^{\lambda q}}}{\|g\|_{p,v}^q}\ge
\frac{\int_0^\infty\! \Bigl(\int_{\xi x}^x g_\xi(z)\,\mathrm{d}z\!\Bigr)^q
  \! \frac{\mathrm{d}x}{x^{\lambda q}}}{\|g_\xi\|_{p,v}^q}\gg F^q(\xi).
\end{equation} In other words, $g_\xi$ is supposed to be a test function corresponding to the chosen type of the functional $F(\xi)$.
Functions of such a type appear in the proof of the estimates $\|\mathcal{H}_\xi\|\gg F(\xi)$ (see Theorem 2.2.(d,e) for details). The required function $f_\xi$ can be easily constructed from $g_\xi$. 
We have from \eqref{second} that for some $0<\xi_0<1$
\begin{align}\label{third}
C^q&\ge\sup_{\substack{\|f'\|\not=0\\f'\geq 0}}\frac{\int_0^\infty \left[\int_0^1 \left(\int_{\xi x}^x f'(z)\,\mathrm{d}z\right)^q
  \frac{\mathrm{d}\xi}{(1-\xi)^{1+\lambda q}}  \right]\frac{\mathrm{d}x}{x^{\lambda q}}}{\|f'\|_{p,v}^q}\nonumber\\&\ge\sup_{\substack{\|f'\|\not=0\\f'\geq 0}}\frac{\int_0^\infty \left[\int_0^{\xi_0} \left(\int_{\xi x}^x f'(z)\,\mathrm{d}z\right)^q
  \frac{\mathrm{d}\xi}{(1-\xi)^{1+\lambda q}}  \right]\frac{\mathrm{d}x}{x^{\lambda q}}}{\|f'\|_{p,v}^q}\nonumber\end{align}\begin{align}
&\ge\sup_{\substack{\|f'\|\not=0\\f'\geq 0}}\frac{\int_0^\infty \left(\int_{\xi_0 x}^x f'(z)\,\mathrm{d}z\right)^q
 \frac{\mathrm{d}x}{x^{\lambda q}}}{\|f'\|_{p,v}^q}\int_0^{\xi_0} \frac{\mathrm{d}\xi}{(1-\xi)^{1+\lambda q}}\nonumber
\\&\ge\frac{\int_0^\infty\! \left(\int_{\xi_0 x}^x f_{\xi_0}'(z)\,\mathrm{d}z
\right)^q
 \frac{\mathrm{d}x}{x^{\lambda q}}}{\|f_{\xi_0}'\|_{p,v}^q}\int_0^{\xi_0}\!\!\!\! \frac{\mathrm{d}\xi}{(1-\xi)^{1+\lambda q}}\nonumber\\&\gg F^q(\xi_0)\biggl(\frac{1}{(1-\xi_0)^{\lambda q}}-1\biggr).\end{align}
Thus and \eqref{first}, in view of $F(\xi)\approx \|\mathcal{H}_\xi\|,$ \begin{equation}\label{main}\sup_{0<\xi<1} \|\mathcal{H}_{\xi}\|\biggl(\frac{1}{(1-\xi)^{\lambda q}}-1\biggr)^{\frac{1}{q}}\ll C\ll \biggl(\int_0^1 \frac{\|\mathcal{H}_\xi\|^q\,\mathrm{d}\xi} {(1-\xi)^{1+\lambda q}}\biggr)^{\frac{1}{q}}.\end{equation}

Notice that the boundaries $a(x)=\xi x$ and $b(x)=x$ of the Hardy--Steklov operator \eqref{u10} satisfy the conditions \eqref{3}. Besides,
the explicit form of the weight function $w(x)=x^{-\lambda q}$ is suitable for constructing the dual fairway $b^{-1}(y)\le\rho(y)\le a^{-1}(y)$
on $(0,\infty)$ with the property \eqref{10}. We find from \eqref{10} that
\begin{equation}\label{l2} \rho(y)=\zeta y:=\zeta(\xi) y:=\begin{cases}2^{-\frac{1}{1-\lambda q}}(1+\xi^{\lambda q-1})^{\frac{1}{1-\lambda q}}y, & \lambda q\not =1,\\ y /\sqrt{\xi}, & \lambda q=1.\end{cases}\end{equation} In order to characterize the inequality \eqref{u=1}
we need the intervals \begin{align*} \Delta_\xi(t)\!\!=&[a(t),b(t)]=[\xi t,t],& \delta_\xi(t)\!\!=&[b^{-1}(\rho^{-1}(t)),a^{-1}(\rho^{-1}(t))]
=\Bigl[\frac{t}{\zeta},\frac{t}{\xi\zeta}\Bigr],\\
\Theta_\xi(t)\!\!=&[b^{-1}(t),a^{-1}(t)]=\Bigl[t,\frac{t}{\xi}\Bigr],& \vartheta_\xi(t)\!\!=&[a(\rho(t)),b(\rho(t))]=[\xi\zeta t,\zeta t].\end{align*}
Using the result of Corollary 2.3 from Section 2 of this paper, we obtain that
\begin{equation*}
\|\mathcal{H}_\xi\|\approx\begin{cases}
\mathcal{A}(\xi)\approx\mathcal{A}^\ast(\xi), &1<p\le q<\infty,\\
\mathcal{B}(\xi)\approx\mathcal{B}_-^\ast(\xi) +\mathcal{B}_+^\ast(\xi), &0<q<p<\infty,\ p>1,
\end{cases}\quad\end{equation*} where  \begin{equation}\label{l1}
\mathcal{A}(\xi)=\begin{cases}
\Bigl(\frac{\xi^{\lambda q-1}-1}{(1-\lambda q) (1+\xi^{\lambda q-1})}\Bigr)^{\frac{1}{q}}
\sup_{t>0} t^{\frac{1}{q}-\lambda} \Bigl(\int_{\xi t}^t \varrho(y)\,\mathrm{d}y \Bigr)^{\frac{1}{p'}}, & \lambda q\not = 1,\\
[\ln 1/\xi]^{\frac{1}{q}}\sup_{t>0}\Bigl(\int_{\xi t}^t \varrho(y)\,\mathrm{d}y\Bigr)^{\frac{1}{p'}}, & \lambda q= 1; \end{cases}\end{equation}\begin{equation*}
\mathcal{A}^\ast(\xi)=\begin{cases} \Bigl(\frac{\xi^{\lambda q-1}-1}{1-\lambda q}\Bigr)^{\frac{1}{q}}\sup_{t>0} t^{\frac{1}{q}-\lambda}
\Bigl(\int_{\xi\zeta t}^{\zeta t} \varrho(y)\,\mathrm{d}y\Bigr)^{\frac{1}{p'}},& \lambda q\not= 1,\\ [\ln 1/\xi]^{\frac{1}{q}}\sup_{t>0}\Bigl(\int_{\xi\zeta t}^{\zeta t} \varrho(y)\,\mathrm{d}y\Bigr)^{\frac{1}{p'}}, & \lambda q= 1;\end{cases}\end{equation*}\begin{equation}\label{l3}
\mathcal{B}(\xi)=\begin{cases}
\Bigl(\frac{\xi^{\lambda q-1}-1}{(1-\lambda q) (1+\xi^{\lambda q-1})}\Bigr)^{\frac{1}{p}}
\Bigl(\int_0^\infty t^{\frac{r}{p}-\lambda r} \Bigl[\int_{\xi t}^t \varrho 
\Bigr]^{\frac{r}{p'}}
\mathrm{d}t\Bigr)^{\frac{1}{r}},& \lambda q\not = 1,\\
[\ln 1/\xi]^{\frac{1}{p}} \Bigl(\int_0^\infty t^{-1} \Bigl[\int_{\xi t}^t \varrho(y)\,\mathrm{d}y\Bigr]^{\frac{r}{p'}}
\mathrm{d}t\Bigr)^{\frac{1}{r}}, & \lambda q=1;\end{cases}\end{equation}\begin{equation*}
\mathcal{B}_-^\ast(\xi)=\begin{cases} \Bigl(\frac{\xi^{\lambda q-1}-1}{1-\lambda q}\Bigr)^{\frac{1}{q}}\Bigl(\int_0^\infty t^{\frac{r}{q}}
\Bigl[\int_{\xi\zeta t}^{
t} \varrho(y)\,\mathrm{d}y\Bigr]^{\frac{r}{q'}}\varrho(t)\,\mathrm{d}t\Bigr)^{\frac{1}{r}},& \lambda q\not = 1, \\ [\ln 1/\xi]^{\frac{1}{q}}
\Bigl(\int_0^\infty \Bigl[\int_{\xi\zeta t}^{
t} \varrho\Bigr]^{\frac{r}{q'}}\varrho(t)\,\mathrm{d}t\Bigr)^{\frac{1}{r}}, &\lambda q=1;\end{cases}\end{equation*}\begin{equation*}
\mathcal{B}_+^\ast(\xi)=\begin{cases} \Bigl(\frac{\xi^{\lambda q-1}-1}{1-\lambda q}\Bigr)^{\frac{1}{q}}\Bigl(\int_0^\infty t^{\frac{r}{q}}
\Bigl[\int_{t}^{\zeta t}  \varrho(y)\,\mathrm{d}y\Bigr]^{\frac{r}{q'}}\varrho(t)\,\mathrm{d}t\Bigr)^{\frac{1}{r}},& \lambda q\not = 1, \\ [\ln 1/\xi]^{\frac{1}{q}}
\Bigl(\int_0^\infty \Bigl[\int_{t}^{\zeta t}  \varrho(y)\,\mathrm{d}y\Bigr]^{\frac{r}{q'}}\varrho(t)\,\mathrm{d}t\Bigr)^{\frac{1}{r}}, &\lambda q=1,\end{cases}\end{equation*} (see functionals \eqref{A}---\eqref{B}). Thus and \eqref{main}, we extract characterizations for the inequality \eqref{00} with $u=1$: if $1<p\le q<\infty$ then \begin{equation}\label{m1}\underline{\mathbb{A}}^\ast \approx\underline{\mathbb{A}}\ll C\ll\bar{\mathbb{A}}\approx\bar{\mathbb{A}}^\ast,\quad\textrm{where}\end{equation} 
 \begin{align*}\label{m3}\underline{ \mathbb{A}}&:=\displaystyle\sup_{0<\xi<1}\mathcal{A}(\xi)
\biggl(\frac{1}{(1-\xi)^{\lambda q}}-1\biggr)^{\frac{1}{q}},&
\bar{ \mathbb{A}}&:=\displaystyle \biggl(\int_0^1 \frac{[\mathcal{A}(\xi)]^q\,
\mathrm{d}\xi}{(1-\xi)^{1+\lambda q}} \biggr)^{\frac{1}{q}},\\
\underline{ \mathbb{A}}^\ast&:=\displaystyle\sup_{0<\xi<1}\mathcal{A}^\ast(\xi)
\biggl(\frac{1}{(1-\xi)^{\lambda q}}-1\biggr)^{\frac{1}{q}},&
\bar{ \mathbb{A}}^\ast&:=\displaystyle \biggl(\int_0^1 \frac{[\mathcal{A}^\ast(\xi)]^q\,
\mathrm{d}\xi}{(1-\xi)^{1+\lambda q}} \biggr)^{\frac{1}{q}};
\end{align*}
for $0<q<p<\infty$ we have \begin{equation} \label{m2}\underline{\mathbb{B}}^\ast \approx\underline{\mathbb{B}}\ll C\ll\bar{\mathbb{B}}\approx\bar{\mathbb{B}}^\ast,\quad\textrm{where}\end{equation} 
\begin{gather*}
\underline{ \mathbb{B}}:=\displaystyle\sup_{0<\xi<1}\mathcal{B}(\xi)
\biggl(\frac{1}{(1-\xi)^{\lambda q}}-1\biggr)^{\frac{1}{q}},\qquad
\bar{ \mathbb{B}}:=\displaystyle \biggl(\int_0^1 \frac{[\mathcal{B}(\xi)]^q\,
\mathrm{d}\xi}{(1-\xi)^{1+\lambda q}} \biggr)^{\frac{1}{q}},\\
\underline{\mathbb{B}}^\ast:=\displaystyle\sup_{0<\xi<1}\Bigl[
(\mathcal{B}_-^\ast(\xi)+\mathcal{B}_+^\ast(\xi)\Bigr]
\biggl(\frac{1}{(1-\xi)^{\lambda q}}-1\biggr)^{\frac{1}{q}},\\
\bar{ \mathbb{B}}^\ast:=\displaystyle \biggl(\int_0^1 \frac{[\mathcal{B}_-^\ast(\xi)+\mathcal{B}_+^\ast(\xi)]^q\,
\mathrm{d}\xi}{(1-\xi)^{1+\lambda q}} \biggr)^{\frac{1}{q}}.\end{gather*}

Characterization of the initial inequality \eqref{00} in its equivalent form \eqref{00'} with general $u$ can be made in the exact same way as for the
one--weighted inequality \eqref{u=1}. Indeed, following the same pattern as the result of the estimates \eqref{first}---\eqref{third}
with general $U_\xi(x),$ and putting $W_\xi(x):=\int_0^\xi U_t(x)\,\mathrm{d}t$, we obtain that
\begin{equation}\label{mainmain}\sup_{0<\xi<1} \|\mathcal{H}_{\xi}\|_{L^p_v\to L^q_{W_\xi}}\ll C\ll \biggl(\int_0^1 \|\mathcal{H}_\xi\|^q_{L^p_v\to L^q_{U_\xi}}\,\mathrm{d}\xi\biggr)^{\frac{1}{q}}.\end{equation}
The estimate \eqref{mainmain} is true for the two--weighted inequality \eqref{00'} as well as for \eqref{00} with general weight function $u$ for all parameters $p>1$ and $q>0$.

To state the main result of this part of the paper we introduce three fairways. Let the first
fairway $\sigma_\xi(x)$ be such that $\xi x<\sigma_\xi(x)<x$, $0<\xi<1$, $x>0$, and
\begin{equation}\label{f1}
\int_{\xi x}^{\sigma_\xi(x)} v^{1-p'}=\int_{\sigma_\xi(x)}^x v^{1-p}\qquad(x>0).
\end{equation} Assume that, given $0<\xi<1$, two others (dual) fairway--functions $\rho_{W_\xi}(y)$ and $\rho_{U_\xi}(y)$ are such that $y<\rho_{W_\xi}(y)<y/\xi$, $y<\rho_{U_\xi}(y)<y/\xi$
for all $y>0$ and
\begin{equation}\label{f2}
\int_{y}^{\rho_{W_\xi}(y)} W_\xi=\int_{\rho_{W_\xi}(y)}^{y/\xi}W_\xi,\quad
\int_{y}^{\rho_{U_\xi}(y)} U_\xi=\int_{\rho_{U_\xi}(y)}^{y/\xi}U_\xi\qquad(y>0).
\end{equation}
\begin{theorem}
Let $p>1, q>0$ and $q\not=1$. Suppose that $C>0$ is the best possible in 
\eqref{00} and independent of $f$. Let $\rho_w$ denote either the fairway--function $\rho_{W_\xi}$ or $\rho_{U_\xi}$ defined by \eqref{f2}. 
Then we have for the case $1<p\le q<\infty$ that
$$\sup_{0<\xi<1}\mathcal{A}_i(W_\xi)=:\underline{\mathbb{A}}(i)\ll C\ll\bar{\mathbb{A}}(i):=\biggl(\int_0^1[\mathcal{A}_i(U_\xi)]^q
\,\mathrm{d}\xi\biggr)^{\frac{1}{q}}\qquad(i=1,2,3,4), $$ where \begin{equation*}\mathcal{A}_1(w):=\sup_{t>0} \biggl(\int_{\sigma_\xi(t)}^{\sigma_\xi(t)/\xi}w(x) \,\mathrm{d}x\biggr)^{\frac{1}{q}}
\biggl(\int_{\xi t}^{t}\varrho(y)\,\mathrm{d}y\biggr)^{\frac{1}{p'}},\end{equation*}
\begin{equation*}\mathcal{A}_2(w):=\sup_{t>0}\biggl(\int_{t}^{t/\xi}w(x) \,\mathrm{d}x\biggr)^{\frac{1}{q}}
\biggl(\int_{\xi\sigma_\xi^{-1} (t)}^{\sigma_\xi^{-1} (t)}\varrho(y)\,\mathrm{d}y\biggr)^{\frac{1}{p'}},\end{equation*}
\begin{equation*}\mathcal{A}_3(w):=\sup_{t>0} \biggl(\int_{\rho_{w}^{-1}(t)}^{\rho_{w}^{-1}(t)/\xi}w(x) \,\mathrm{d}x\biggr)^{\frac{1}{q}}
\biggl(\int_{\xi t}^{t}\varrho(y)\,\mathrm{d}y\biggr)^{\frac{1}{p'}},\end{equation*}
\begin{equation*}\mathcal{A}_4(w):=\sup_{t>0}\biggl(\int_{t}^{t/\xi}w(x) \,\mathrm{d}x\biggr)^{\frac{1}{q}}
\biggl(\int_{\xi\rho_{w}(t)}^{\rho_{w}(t)} \varrho(y)\,\mathrm{d}y\biggr)^{\frac{1}{p'}}.\end{equation*} Moreover, $\underline{\mathbb{A}}(1)\approx\underline{\mathbb{A}}(2)\approx\underline{\mathbb{A}}(3) \approx\underline{\mathbb{A}}(4)$ and $\bar{\mathbb{A}}(1)\approx\bar{\mathbb{A}}(2)\approx\bar{\mathbb{A}}(3) \approx\bar{\mathbb{A}}(4)$.\\ If $0<q<p<\infty$ and $p>1$ then
$$\sup_{0<\xi<1}\mathcal{B}_i(W_\xi)=:\underline{\mathbb{B}}(i)\ll C\ll\bar{\mathbb{B}}(i):=\biggl(\int_0^1[\mathcal{B}_i(U_\xi)]^q
\,\mathrm{d}\xi\biggr)^{\frac{1}{q}}\qquad(i=1,2,3), $$ where
\begin{align*}\mathcal{B}_1(w)&:=\biggl(\int_0^\infty \biggl[\int_{\sigma_\xi(t)}^{\sigma_\xi(t)/\xi}w(x) \,\mathrm{d}x\biggr]^{\frac{r}{p}}
\biggl[\int_{\xi t}^{t}\varrho(y)\,\mathrm{d}y\biggr]^{\frac{r}{p'}} w(t)\,\mathrm{d}t\biggr)^{\frac{1}{r}},\\
\mathcal{B}_2(w)&:=\biggl(\int_0^\infty \biggl[\int_{t}^{t/\xi}w(x) \,\mathrm{d}x\biggr]^{\frac{r}{q}}
\biggl[\int_{\xi\sigma_\xi^{-1} (t)}^{\sigma_\xi^{-1} (t)}\varrho(y)\,\mathrm{d}y\biggr]^{\frac{r}{q'}}\varrho(t)\, \mathrm{d}t\biggr)^{\frac{1}{r}},\\
\mathcal{B}_3(w)&:=\biggl(\int_0^\infty \biggl[\int_{\rho_{w}^{-1}(t)}^{\rho_{w}^{-1}(t)/\xi}w(x) \,\mathrm{d}x\biggr]^{\frac{r}{p}}
\biggl[\int_{\xi t}^{t}\varrho(y)\,\mathrm{d}y\biggr]^{\frac{r}{p'}} w(t)\, \mathrm{d}t\biggr)^{\frac{1}{r}},\end{align*} and \begin{multline*}\underline{\mathbb{B}}(4):=\sup_{0<\xi<1}\Bigl[ \mathcal{B}_4^-(W_\xi)+\mathcal{B}_4^+(W_\xi)^+\Bigr]\ll C\\\ll \bar{\mathbb{B}}(4):= \biggl(\int_0^\infty\Bigl[\mathcal{B}_4^-(U_\xi)+\mathcal{B}_4^+(U_\xi)^+\Bigr]^q\, \mathrm{d}\xi\biggr)^{\frac{1}{q}},\end{multline*} where
\begin{align*}\mathcal{B}_4^-(w)&:=\biggl(\int_0^\infty
\biggl[\int_{t}^{t/\xi}w(x) \,\mathrm{d}x\biggr]^{\frac{r}{q}}
\biggl[\int_{\xi\rho_{w}(t)}^{t} \varrho(y)\,\mathrm{d}y\biggr]^{\frac{r}{q'}}\varrho(t)\, \mathrm{d}t\biggr)^{\frac{1}{r}},\\\mathcal{B}_4^+(w)&:=\biggl(\int_0^\infty
\biggl[\int_{t}^{t/\xi}w(x) \,\mathrm{d}x\biggr]^{\frac{r}{q}}
\biggl[\int_t^{\rho_{w}(t)}\varrho(y)\,\mathrm{d}y\biggr]^{\frac{r}{q'}}\varrho(t)\, \mathrm{d}t\biggr)^{\frac{1}{r}}.\end{align*} Moreover, $\underline{\mathbb{B}}(1)\approx\underline{\mathbb{B}}(2)\approx\underline{\mathbb{B}}(3) \approx\underline{\mathbb{B}}(4)$ and $\bar{\mathbb{B}}(1)\approx\bar{\mathbb{B}}(2)\approx\bar{\mathbb{B}}(3) \approx\bar{\mathbb{B}}(4)$.
\end{theorem} \begin{proof} Theorem follows by \eqref{mainmain} and Corollary 2.3.\end{proof}

To characterize the inequality \eqref{00} for $p=1$ or $q=1$ one can use either Theorem 4 from \cite[Ch. 7, \S~1.5]{KA} or Theorem 3.1 from this paper with $p\to 1$ or $q\to 1$. In particular, on the strength
of Corollary 2.3,
$$\lim_{p\to 1}\mathcal{A}_{\rho_w^{-1}}=\sup_{t>0} \biggl(\int_{\rho_w^{-1}(t)}^{\rho_w^{-1}(t)/\xi}w(x)\,\mathrm{d}x\biggr)^{\frac{1}{q}}
\esup_{\xi t<y<t}v^{-1}(y)=:A(w),$$
$$\lim_{p\to 1}(\mathcal{A}_{\rho_w})^\ast=\sup_{t>0} \biggl(\int_{t}^{t/\xi}w(x)\,\mathrm{d}x\biggr)^{\frac{1}{q}} 
\esup_{\xi \rho_w(t)<y<\rho_w(t)}v^{-1}(y)=:A^\ast(w),$$
$$\lim_{p\to 1}\mathcal{B}_{\rho_w^{-1}}=\biggl(\int_0^\infty \biggl[\int_{\rho_w^{-1}(t)}^{\rho_w^{-1}(t)/\xi}w
\biggr]^ {\frac{1}{1-q}}[
\esup_{\xi t<y<t}v^{-1}(y)]^{\frac{q}{1-q}}\,w(t)\,\mathrm{d}t\biggr) ^{\frac{1}{q}-1}=:B(w).$$ Thus,
$$ \sup_{0<\xi<1}
A(W_\xi)\approx\!\sup_{0<\xi<1}\!
A^\ast(W_\xi)\ll C\!\ll\!\biggl(\int_0^1\!
[A(U_\xi)]^q\,\mathrm{d}\xi\biggr)^{\frac{1}{q}}\!\!\approx\!\biggl(\int_0^1\!
[A^\ast(U_\xi)]^q\,\mathrm{d}\xi\biggr)^{\frac{1}{q}}$$ if $p=1\le q<\infty$, and
$$\sup_{0<\xi<1}
B^\ast(W_\xi)\ll C\ll\biggl(\int_0^1
[B(U_\xi)]^q\,\mathrm{d}\xi\biggr)^{\frac{1}{q}}\qquad(0<q<p=1).$$ Analogously, one can characterize \eqref{00} if $q=1$, using even the both types of the functionals
$\mathcal{A}_\varsigma$, $(\mathcal{A}_{\varsigma^\ast})^\ast$, $\mathcal{B}_\varsigma$, $(\mathcal{B}_{\varsigma^\ast})^\ast$
(with $\varsigma=\sigma$ or $\varsigma=\rho^{-1}$) when $p'\not=\infty$. If $p'=\infty$ then only the functionals
$\mathcal{A}_{\rho^{-1}}$, $(\mathcal{A}_{\rho})^\ast$ and $\mathcal{B}_{\rho^{-1}}$ can be used. Notice also that if $q=1$ then $r/q'=0$, and,
therefore,
$$\lim_{q\to 1}(\mathcal{B}_{\sigma_\xi^{-1}})^\ast=\lim_{q\to 1}(\mathcal{B}_{\rho_w})^\ast=
\biggl(\int_0^\infty\biggl[\int_{b^{-1}(t)}^{a^{-1}(t)}w
\biggr]^{p'}
\varrho(t)\,\mathrm{d}t\biggr)^{\frac{1}{p'}},\qquad 1<p<\infty.$$

We complete the section by two examples giving criteria for \eqref{00} to hold with particular $u$ and $v$.
\begin{example} Let $1<p\le q<\infty$, $u=1$ and $v(z)=z^\alpha$. 
In order to derive necessary and sufficient conditions for the inequality \eqref{00} to hold under these assumptions we shall start from calculating the functional $\mathcal{A}(\xi)$ of the form \eqref{l1}: \begin{equation*}
\mathcal{A}(\xi)=\begin{cases}
\Bigl(\frac{\xi^{\lambda q-1}-1}{(1-\lambda q) (1+\xi^{\lambda q-1})}\Bigr)^{\frac{1}{q}}
\sup_{t>0} t^{\frac{1}{q}-\lambda} \Bigl(\int_{\xi t}^t z^{\alpha(1-p')}\,\mathrm{d}z\Bigr)^{\frac{1}{p'}}, & \lambda q\not = 1,\\
[\ln 1/\xi]^{\frac{1}{q}}\sup_{t>0}\Bigl(\int_{\xi t}^t z^{\alpha(1-p')}\,\mathrm{d}z \Bigr)^{\frac{1}{p'}}, & \lambda q= 1. \end{cases}
\end{equation*} The suprema in $\mathcal{A}(\xi)$ are finite if and only if \begin{equation}\label{m4}
\frac{1}{q}-\lambda+\left(\alpha(1-p')+1\right)\frac{1}{p'}=0. \end{equation} In such a case $\alpha(1-p')+1=p'(\lambda-1/q)$ and, therefore,
\begin{equation*}
\mathcal{A}(\xi)=\begin{cases}
\Bigl(\frac{\xi^{\lambda q-1}-1}{(1-\lambda q) (1+\xi^{\lambda q-1})}\Bigr)^{\frac{1}{q}}
\Bigl(\frac{1-\xi^{p'(\lambda-1/q)}}{p'(\lambda-1/q)}\Bigr)^{\frac{1}{p'}}, & \lambda q\not = 1,\\
[\ln 1/\xi]^{\frac{1}{q}+\frac{1}{p'}}, & \lambda q= 1. \end{cases}
\end{equation*}
It follows from \eqref{m1} that the inequality \eqref{00} is satisfied if $\bar{ \mathbb{A}}<\infty$, where \begin{equation*}\bar{ \mathbb{A}}=\begin{cases}
\Bigl(\int_0^1\frac{\xi^{\lambda q-1}-1}{(1-\lambda q) (1+\xi^{\lambda q-1})}
\Bigl(\frac{1-\xi^{p'(\lambda-1/q)}}{p'(\lambda-1/q)}\Bigr)^{\frac{q}{p'}}\frac {\mathrm{d}\xi}{(1-\xi)^{1+\lambda q}} \Bigr)^{\frac{1}{q}}, & \lambda q\not =1,\\
\Bigl(\int_0^1[\ln 1/\xi]^{1+\frac{q}{p'}}\frac {\mathrm{d}\xi}{(1-\xi)^{2}} \Bigr)^{\frac{1}{q}}, & \lambda q=1;\end{cases}\end{equation*}
and if \eqref{00} holds with $C<\infty$ then $\underline{\mathbb{A}}<\infty$, where
\begin{equation*}\underline{\mathbb{A}}=\begin{cases}
\sup\limits_{0<\xi<1}\Bigl(\frac{\xi^{\lambda q-1}-1}{(1-\lambda q) (1+\xi^{\lambda q-1})}\Bigr)^{\frac{1}{q}}
\Bigl(\frac{1-\xi^{p'(\lambda-1/q)}}{p'(\lambda-1/q)}\Bigr)^{\frac{1}{p'}}
\Bigl(\frac{1}{(1-\xi)^{\lambda q}}-1\Bigr)^{\frac{1}{q}}, & \lambda q\not =1,\\  \sup\limits_{0<\xi<1} [\ln 1/\xi]^{\frac{1}{q}+\frac{1}{p'}}\,(1-\xi)^{-\lambda}\,\xi^{\frac{1}{q}} , & \lambda q= 1.\end{cases}\end{equation*}
We write
 \begin{equation*}\bar{ \mathbb{A}}\simeq\begin{cases}
\Bigl(\int_0^1\frac{1-\xi^{\lambda q-1}}{1+\xi^{\lambda q-1}}
\, [1-\xi^{p'(\lambda-1/q)}]^{\frac{q}{p'}}\frac {\mathrm{d}\xi}{(1-\xi)^{1+\lambda q}} \Bigr)^{\frac{1}{q}}, & \lambda q>1,\\
\Bigl(\int_0^1\frac{1-\xi^{1-\lambda q}}{\xi^{1-\lambda q}+1}\,
[1-\xi^{p'(1/q-\lambda)}]^{\frac{q}{p'}}\,\frac {\xi^{\lambda q-1}\,\mathrm{d}\xi}{(1-\xi)^{1+\lambda q}} \Bigr)^{\frac{1}{q}}, & \lambda q<1,\\
\Bigl(\int_0^1[\ln 1/\xi]^{1+\frac{q}{p'}}\frac {\mathrm{d}\xi}{(1-\xi)^{2}} \Bigr)^{\frac{1}{q}}, & \lambda q=1;\end{cases}\end{equation*} and
\begin{equation*}\underline{\mathbb{A}}\simeq\begin{cases}
\sup\limits_{0<\xi<1}\Bigl(\frac{1-\xi^{\lambda q-1}}{1+\xi^{\lambda q-1}}\Bigr)^{\frac{1}{q}}
[1-\xi^{p'(\lambda-1/q)}]^{\frac{1}{p'}}(1-\xi)^{-\lambda}
[1-(1-\xi)^{\lambda q}]^{\frac{1}{q}} & \lambda q>1,\\
\sup\limits_{0<\xi<1}\Bigl(\frac{1-\xi^{1-\lambda q}}{\xi^{1-\lambda q}+1}\Bigr)^{\frac{1}{q}}
[1-\xi^{p'(1/q-\lambda)}]^{\frac{1}{p'}}(1-\xi)^{-\lambda}\xi^{\lambda-1/q}
[1-(1-\xi)^{\lambda q}]^{\frac{1}{q}}, & \lambda q<1,\\  \sup\limits_{0<\xi<1} [\ln 1/\xi]^{\frac{1}{q}+\frac{1}{p'}}\,(1-\xi)^{-\lambda}\,\xi^{\frac{1}{q}} , & \lambda q= 1.\end{cases}\end{equation*}
Notice that for some $\beta>0$ \begin{equation}\label{eqv}
1-\xi^\beta\approx \xi^{\beta -1}(1-\xi)\quad\textrm{if}\ \ 1/2\le\xi\le 1
\end{equation} and \begin{equation}\label{eqvi}
1-(1-\xi)^\beta\approx (1-\xi)^{\beta -1}\xi\quad\textrm{if}\ \ 0\le\xi\le 1/2.
\end{equation} If $\lambda>1/q$ then the integral in $\bar{ \mathbb{A}}$ converges at the point $0$, and, by \eqref{eqv}, converges at $1$ if $q/p'-\lambda q>-1$. Therefore, $\bar{ \mathbb{A}}<\infty$ in the case $\lambda>1/q$ if \begin{equation}\label{usl}\lambda<\frac{1}{p'}+\frac{1}{q}.\end{equation} This condition 
is automatically satisfied and is also sufficient for finiteness of $\bar{ \mathbb{A}}$ if $\lambda<1/q$. If $\lambda=1/q$ then, by substitution $\xi=\mathrm{e}^{-t}$, the functional $\bar{ \mathbb{A}}$ takes the form $$\bar{ \mathbb{A}}=\biggl(\int_0^\infty\frac{t^{1+q/p'}\mathrm{e}^{-t}\mathrm{d}t} {(1-\mathrm{e}^{-t})^{2}} \biggr)^{\frac{1}{q}}.$$ This integral converges at $\infty$. Since \begin{equation}\label{l6} \lim_{t\to 0}t/(1-\mathrm{e}^{-t})=1\end{equation} the integrand behaves like $t^{q/p'-1}$ near $0$, so $\bar{ \mathbb{A}}$ is finite. Thus, the condition \eqref{usl} is sufficient for the inequality \eqref{00} to hold when $u=1$ and $v(z)=z^\alpha$ with $$ \alpha:=\left(\frac{1}{q}-\lambda+1\right)p-1>0.$$

If $\lambda>1/q$ then the supremum in $\underline{ \mathbb{A}}$ is finite near $0$, and, by \eqref{eqv}, is finite near $1$ if $1/q+1/p'-\lambda\ge 0$. Therefore, if \eqref{00} is true for $\lambda>1/q$ with $C<\infty$ then \begin{equation}\label{uslo}\lambda\le\frac{1}{p'}+\frac{1}{q}.\end{equation} In view of \eqref{eqvi}, the supremum in $\underline{ \mathbb{A}}$ is finite near $0$ in the case $\lambda<1/q$. By \eqref{eqv} the condition \eqref{uslo} necessarily follows from finiteness of the supremum in $\underline{ \mathbb{A}}$ near $1$, and is automatically satisfied if $\lambda<1/q$. If $\lambda=1/q$ then, substituting $\xi=\mathrm{e}^{-t}$, we arrive to 
$$\underline{ \mathbb{A}}=\sup_{t>0}\frac{t^{\frac{1}{q}+\frac{1}{p'}}\mathrm{e}^{-\frac{t}{q}}} {(1-\mathrm{e}^{-t})^{\lambda}}.$$ This supremum is finite near $\infty$. In view of \eqref{l6} and $\lambda=1/q$ it is finite near $0$ as well.

Thus, the condition \eqref{uslo} is necessary for \eqref{00} when $u=1$ and $v(z)=z^\alpha$, $\alpha\ge 0.$ Combining this with \eqref{usl} we obtain the following criterion. \begin{proposition} Let $1<p\le q<\infty$, $0<\lambda<1$ and put $ \alpha=(1/q-\lambda+1)p-1.$ Assume that $u=1$ and $v(z)=z^\alpha$ in the inequality \eqref{00}. Then \eqref{00} holds if and only if $\lambda<1/p'+1/q$. \end{proposition}
\begin{corollary}
In particular, if $p=q$, then the inequality
$$\int_0^\infty\!\!\!\!\int_0^\infty\frac{|f(x)-f(y)|^p}{|x-y|^{1+\lambda
p}}\,\mathrm{d}x\,\mathrm{d}y\le C\int_0^\infty |f'(z)|^p z^{(1-\lambda)p}\,\mathrm{d}z$$ is valid if and only if $0<\lambda<1$.
\end{corollary}\begin{remark} The suffucient part of Corollary 2.4 was proven in \cite[pp. 175--176]{HS}.\end{remark} 
\end{example}

\begin{example} Let $0<q<p<\infty$, $u=1$ and $v(z)=z^{\frac{\gamma-1}{1-p'}}(1+z^\gamma)^{\frac{2}{p'-1}}$, $\gamma\not =0$.  We shall find boundaries on $\lambda(\gamma)$ following from characteristics on \eqref{00} for this case. We shall use the functional $\mathcal{B}(\xi)$ of the form \eqref{l3}.  Since 
$$\int_{\xi t}^t v^{1-p'}(z)\,\mathrm{d}z=-
\frac{1}{\gamma(1+z^\gamma)}\biggl|_{\xi t}^t=\frac{t^{|\gamma|}(1-\xi^{|\gamma|})}
{{|\gamma|}[1+(\xi t)^{{|\gamma|}}](1+t^{|\gamma|})}$$ for any $\gamma\not= 0$ then we may and shall assume that $\gamma>0$. We have
$$\mathcal{B}(\xi)\simeq\begin{cases}
\Bigl(\frac{(\xi^{\lambda q-1}-1)(1-\xi^\gamma)^{p-1}}{(1-\lambda q) (1+\xi^{\lambda q-1})}\Bigr)^{\frac{1}{p}}
\Bigl(\int_0^\infty \frac{t^{r/p-\lambda r+\gamma r/p'}\,\mathrm{d}t}{[1+(\xi t)^{\gamma}]^{r/p'}[1+t^\gamma]^{r/p'}}
\Bigr)^{\frac{1}{r}},& \lambda q\not = 1,\\
[\ln 1/\xi]^{\frac{1}{p}}(1-\xi^\gamma)^{\frac{1}{p'}} \Bigl(\int_0^\infty \frac{t^{\gamma r/p'-1}\,\mathrm{d} t}{[1+(\xi t)^{\gamma}]^{r/p'}[1+t^\gamma]^{r/p'}}
\Bigr)^{\frac{1}{r}}, & \lambda q=1.\end{cases}$$ Put $s:=r(1/q-\lambda)+\gamma r/p'$ and assume that $\mathcal{B}(\xi)<\infty$. Then the integral $$I(\xi):=\int_0^\infty \frac{t^{s-1}\,\mathrm{d}t}{[1+(\xi t)^{\gamma}]^{r/p'}[1+t^\gamma]^{r/p'}}$$ in $\mathcal{B}(\xi)$ converges at $0$. Therefore, $s>0$, that is our $\lambda$ has to satisfy \begin{equation} \label{l4}\lambda< \frac{1}{q}+\frac{\gamma}{p'}
.\end{equation} The condition \eqref{l4} also guarantees convergence $I(\xi)$ at $\infty$ in the case $\lambda q>1$. Indeed, since $\frac{1}{(1+cz)(1+z)}=\frac{1}{1-c}\Bigl[\frac{1}{1+z}-\frac{c}{1+cz}\Bigr]$, then
\begin{multline}I(\xi)\approx\int_0^\infty \frac{t^{s-1}\,\mathrm{d}t}{[1+\xi^{\frac{\gamma r}{p'}}t^{\frac{\gamma r}{p'}}][1+t^{\frac{\gamma r}{p'}}]}=\frac{1/s}{1-\xi^{\frac{\gamma r}{p'}}}\int_0^\infty\biggl[\frac{\mathrm{d}t^s}{1+t^{\frac{\gamma r}{p'}}}- \frac{\xi^{\frac{\gamma r}{p'}}\mathrm{d}t^s}{1+\xi^{\frac{\gamma r}{p'}}t^{\frac{\gamma r}{p'}}}\biggr]\\ \approx\frac{1/s}{1-\xi^{\frac{\gamma r}{p'}}}\biggl[\int_0^\infty\frac{\mathrm{d}[t^s+1]}{[1+t^s]^{\frac{\gamma r}{p's}}}- \xi^{\frac{\gamma r}{p'}-s}\int_0^\infty \frac{\mathrm{d}[\xi^st^s+1]}{[1+\xi^st^s]^{\frac{\gamma r}{p's}}}\biggr]\\=\frac{1-\xi^{r(\lambda-1/q)}}{s(1-\xi^{\frac{\gamma r}{p'}})}\label{l7}
\end{multline} in view of $s>0$ and $\lambda q>1$. If $\lambda q=1$ then $s=\gamma r/p'$ and , analogously to \eqref{l7}, $$I(\xi)\approx \frac{1}{1-\xi^{\gamma r/p'}}\ln[1/\xi].$$ If $\lambda q<1$, then, by $\frac{z}{(1+cz)(1+z)}=\frac{1}{1-c}\Bigl[\frac{1}{1+cz}-\frac{1}{1+z}\Bigr]$ and in view of $r/p+1=r/q$,
\begin{multline*}I(\xi)\approx\int_0^\infty \frac{t^{\frac{\gamma r}{p'}}\cdot t^{\frac{r}{p}-\lambda r}\,\mathrm{d}t}{[1+\xi^{\frac{\gamma r}{p'}}t^{\frac{\gamma r}{p'}}][1+t^{\frac{\gamma r}{p'}}]}=\frac{1/(\frac{r}{q}-\lambda r)}{1-\xi^{\frac{\gamma r}{p'}}}\int_0^\infty\biggl[\frac{\mathrm{d}t^{\frac{r}{q}-\lambda r}}{1+\xi^{\frac{\gamma r}{p'}}t^{\frac{\gamma r}{p'}}}- \frac{\mathrm{d}t^{\frac{r}{q}-\lambda r}}{1+t^{\frac{\gamma r}{p'}}}\biggr]\\ \approx\frac{1/(\frac{r}{q}-\lambda r)}{1-\xi^{\frac{\gamma r}{p'}}}\biggl[\frac{1}{\xi^{\frac{r}{q}-\lambda r}}\int_0^\infty\frac{\mathrm{d}[\xi^{\frac{r}{q}-\lambda r}t^{\frac{r}{q}-\lambda r}+1]}{[1+\xi^{\frac{r}{q}-\lambda r}t^{\frac{r}{q}-\lambda r}]^{\frac{\gamma}{p'(1/q-\lambda)}}}\\- \int_0^\infty \frac{\mathrm{d}[t^{\frac{r}{q}-\lambda r}+1]}{[1+t^{\frac{r}{q}-\lambda r}]^{\frac{\gamma}{p'(1/q-\lambda)}}}\biggr].
\end{multline*}
If we suppose that $\frac{\gamma}{p'(1/q-\lambda)}>1$, that is  \begin{equation} \label{l5}\lambda>\frac{1}{q}-\frac{\gamma}{p'}
,\end{equation} then $I(\xi)$ converges in the case $\lambda q<1$ and
$$I(\xi)\approx \frac{1/(\frac{r}{q}-\lambda r)}{1-\xi^{\frac{\gamma r}{p'}}}\biggl[\frac{1}{\xi^{\frac{r}{q}-\lambda r}}-1\biggr].$$
Now, under assumptions \eqref{l4} and \eqref{l5}, we obtain that  
$$\mathcal{B}(\xi)\approx\begin{cases}
\Bigl(\frac{(1-\xi^{\lambda q-1})(1-\xi^\gamma)^{p-1}}{1+\xi^{\lambda q-1}}\Bigr)^{\frac{1}{p}}\Bigl(\frac{1-\xi^{r(\lambda-1/q)}}{1-\xi^{\frac{\gamma r}{p'}}}\Bigr)^{\frac{1}{r}}
,& \lambda q> 1,\\
\Bigl(\frac{(1-\xi^{1-\lambda q})(1-\xi^\gamma)^{p-1}}{1+\xi^{1-\lambda q}}\Bigr)^{\frac{1}{p}}\xi^{-(\frac{1}{q}-\lambda)} \Bigl(\frac{1-\xi^{r(1/q-\lambda)}}{1-\xi^{\frac{\gamma r}{p'}}}\Bigr)^{\frac{1}{r}}
,& \lambda q< 1,\\
[\ln 1/\xi]^{\frac{1}{q}}(1-\xi^\gamma)^{\frac{1}{p'}} \Bigl(1-\xi^{\frac{\gamma r}{p'}}\Bigr)^{-\frac{1}{r}}, & \lambda q=1.\end{cases}$$

It follows from \eqref{m2} that the inequality \eqref{00} is fulfilled if $\bar{ \mathbb{B}}<\infty$, where \begin{equation*}\bar{ \mathbb{B}}\approx\begin{cases}
\Bigl(\int_0^1
\Bigl[\frac{(1-\xi^{\lambda q-1})(1-\xi^\gamma)^{p-1}}{1+\xi^{\lambda q-1}}\Bigr]^{\frac{q}{p}}\Bigl[\frac{1-\xi^{r(\lambda-1/q)}}{1-\xi^{\frac{\gamma r}{p'}}}\Bigr]^{\frac{q}{r}}\frac {\mathrm{d}\xi}{(1-\xi)^{1+\lambda q}} \Bigr)^{\frac{1}{q}}, & \lambda q>1,\\\Bigl(\int_0^1
\Bigl[\frac{(1-\xi^{1-\lambda q})(1-\xi^\gamma)^{p-1}}{1+\xi^{1-\lambda q}}\Bigr]^{\frac{q}{p}}\xi^{\lambda q-1} \Bigl[\frac{1-\xi^{r(1/q-\lambda)}}{1-\xi^{\frac{\gamma r}{p'}}}\Bigr]^{\frac{q}{r}}\frac {\mathrm{d}\xi}{(1-\xi)^{1+\lambda q}} \Bigr)^{\frac{1}{q}}, & \lambda q<1,\\
\Bigl(\int_0^1\ln 1/\xi[1-\xi^\gamma]^{\frac{q}{p'}}\Bigl[1-\xi^{\frac{\gamma r}{p'}}\Bigr]^{-\frac{q}{r}}\frac {\mathrm{d}\xi}{(1-\xi)^{2}} \Bigr)^{\frac{1}{q}}, & \lambda q=1;\end{cases}\end{equation*}
and if \eqref{00} holds with $C<\infty$ then $\underline{\mathbb{B}}<\infty$, where
\begin{equation*}\underline{\mathbb{B}}\approx\begin{cases}
\sup\limits_{0<\xi<1}\Bigl(\frac{(1-\xi^{\lambda q-1})(1-\xi^\gamma)^{p-1}}{1+\xi^{\lambda q-1}}\Bigr)^{\frac{1}{p}}\Bigl(\frac{1-\xi^{r(\lambda-1/q)}}{1-\xi^{\frac{\gamma r}{p'}}}\Bigr)^{\frac{1}{r}}
\Bigl(\frac{1}{(1-\xi)^{\lambda q}}-1\Bigr)^{\frac{1}{q}}, & \lambda q>1,\\  
\sup\limits_{0<\xi<1}\Bigl(\frac{(1-\xi^{1-\lambda q})(1-\xi^\gamma)^{p-1}}{1+\xi^{1-\lambda q}}\Bigr)^{\frac{1}{p}}\xi^{-(\frac{1}{q}-\lambda)} \Bigl(\frac{1-\xi^{r(1/q-\lambda)}}{1-\xi^{\frac{\gamma r}{p'}}}\Bigr)^{\frac{1}{r}}
\Bigl(\frac{1}{(1-\xi)^{\lambda q}}-1\Bigr)^{\frac{1}{q}}, & \lambda q<1,\\\sup\limits_{0<\xi<1}[\ln 1/\xi]^{\frac{1}{q}}(1-\xi^\gamma)^{\frac{1}{p'}} \Bigl(1-\xi^{\frac{\gamma r}{p'}}\Bigr)^{-\frac{1}{r}} (1-\xi)^{-\lambda}\xi^{\frac{1}{q}} , & \lambda q= 1.\end{cases}\end{equation*} If $\lambda q\not =1$ then the integrals in $\bar{\mathbb{B}}$ converge at 0, and, by \eqref{eqv}, converge at $1$ because of $\lambda<1$. If $\lambda q=1$, then, by substitution $\xi=\mathrm{e}^{-t}$, the integral in $\bar{\mathbb{B}}$ takes the form
$$\int_0^\infty\frac{t(1-\mathrm{e}^{-\gamma t})^{\frac{q}{p'}}\mathrm{e}^{-t}\,\mathrm{d}t}{(1-\mathrm{e}^{-\frac{\gamma r}{p'} t})^{\frac{q}{r}}(1-\mathrm{e}^{-t})^{2}}.$$ In view of \eqref{l6}  the integrand behaves like $t^{q-2}$ near $0$, therefore, the integral converges if $q\ge 1$, which correlates with $q=1/\lambda>1$.

Further, if $\lambda q\not =1$ then the supremums in $\underline{\mathbb{B}}$ are finite near $0$ and $1$ in view of \eqref{eqvi} and \eqref{eqv}, respectively. In the case $\lambda q=1$ the substitution $\xi=\mathrm{e}^{-t}$ and the relation \eqref{l6} yield finiteness of the supremum in $\underline{\mathbb{B}}$.

Combining all the above estimates we can state the following Proposition.
\begin{proposition} Let $0<q<p<\infty$, $0<\lambda<1$ and $\gamma>0.$ Assume that $u=1$ and $v(z)=z^{\frac{\gamma-1}{1-p'}}(1+z^\gamma)^{\frac{2}{p'-1}}$ in the inequality \eqref{00}. Then \eqref{00} holds if and only if \begin{equation*}\label{l8}\Bigl|\lambda-\frac{1}{q}\Bigr|<\frac{\gamma}{p'}. \end{equation*} 
\end{proposition}

\end{example}

{\bf Acknowledgments.} The authors thank all the members of the seminar conducted weekly at the Laboratory of approximate methods and functional analysis in the Computing Center of FEB RAS in Khabarovsk for productive discussions and valuable remarks.


%
%




\end{document}